\documentclass[11pt]{article}

\usepackage{lmodern}
\usepackage{amsmath,amssymb,amsthm,amsfonts}
\usepackage{geometry}
\usepackage[hidelinks]{hyperref}
\hypersetup{pdfauthor={Wei Xie},
pdftitle={Exponent Profiles of Oscillatory Matrices under Matrix and Inverse Sparsity}}

\theoremstyle{plain}
\newtheorem{theorem}{Theorem}[section]
\newtheorem{lemma}[theorem]{Lemma}
\newtheorem{corollary}[theorem]{Corollary}
\newtheorem{proposition}[theorem]{Proposition}

\theoremstyle{definition}
\newtheorem{definition}[theorem]{Definition}
\newtheorem{example}[theorem]{Example}

\theoremstyle{remark}

\numberwithin{equation}{section}

\title{Exponent Profiles of Oscillatory Matrices under Matrix and Inverse Sparsity}
\author{Wei Xie}
\date{}

\begin{document}
\maketitle

\begin{abstract}
We classify the simultaneous half-bandwidths \(p,q\) of a symmetric
oscillatory matrix \(A\) and its inverse. At order \(n\ge2\), the
possible pairs are \((1,n-1)\), \((n-1,1)\), and all pairs with
\(2\le p,q\le n-1\).

Let \(e_k(A)\) be the least positive integer \(m\) for which all
minors of order \(k\) of \(A^m\) are positive, and let
\(e(A)=\max_{1\le k<n}e_k(A)\). Put \(a=\min\{p,q\}\) and
\(b=\max\{p,q\}\). At every feasible pair of actual half-bandwidths,
\[
\max_A e(A)=
\max\left\{\left\lfloor\frac n2\right\rfloor,
n-a-\left\lfloor\frac{\max\{n-p-q,0\}}{b}\right\rfloor\right\}.
\]
We also determine the maximum of \(\sum_{k=1}^{n-1}e_k(A)\) for every
pair. When \(p+q\ge n\), one realization maximizes all \(e_k\)
simultaneously.

Write \(A=U^T\Delta U\), with \(U\) unit upper
triangular and totally nonnegative, and \(\Delta\) positive diagonal.
Let \(\ell(U)\) be the least number of elementary upper bidiagonal
factors with positive parameters whose product is \(U\).
At fixed \((p,q)\), its minimum is \(\max\{n-1,p+q-1\}\).
One symmetric integer matrix of determinant one simultaneously attains
this minimum, the maximum exponent, and the maximum exponent sum.

We prove
\(\ell(U^{r+1})-\ell(U^r)=\#\{k:1\le k<n,\ e_k(A)>r\}\)
for every integer \(r\ge\lfloor(n-1)/2\rfloor\), with the starting point optimal.
With \(\delta=\ell(U)-n+1\), the exponent sum satisfies
\(\delta\le n(n-1)/2-\sum_{k=1}^{n-1}e_k(A)
\le\lfloor(n-1)/2\rfloor\delta\), with both constants sharp.
The lower deficit bound is attained at every admissible factor count.
Within the symmetric class, a total sum of \(n(n-1)/2\) characterizes
basic matrices.

In the basic oscillatory class, without a symmetry assumption,
a positive integer vector \((f_1,\ldots,f_{n-1})\) is realizable
if and only if \(f_k\le\max\{k,n-k\}\) and
\(f_k+f_{n-k}\ge n\) for \(1\le k<n\), and
\(\lvert f_{k+1}-f_k\rvert\le1\) for \(1\le k<n-1\).
Explicit factor orders
realize every such vector.
\end{abstract}



\section{Introduction}
\label{sec:1}

Throughout, matrices are real and have order \(n\ge2\).  A matrix is
\emph{totally nonnegative} (TN) if all its minors are nonnegative and
\emph{totally positive} (TP) if all its minors are positive.  A square TN
matrix \(A\) is oscillatory if some positive integer power of \(A\) is
TP.  The least positive integer \(m\) for which \(A^m\) is TP is called
the \emph{exponent of \(A\)}, denoted by
\[
e(A)=\min\{m\ge1:A^m\text{ is TP}\}.
\]
Gantmacher and Krein proved that \(e(A)\le n-1\), together with the
associated spectral and variation-diminishing results \cite{GantmacherKrein}.  General
accounts of total positivity are given in
\cite{Karlin,FallatJohnson,Pinkus}.

The exponent \(e(A)\) depends on the zero minors of the matrix.  Fallat used
corner minors and elementary bidiagonal factorizations to study oscillatory
powers \cite{FallatRemark}.  Fallat and Liu characterized the matrices
for which \(e(A)=n-1\) \cite[Theorem~10]{FallatLiu}.  Zarai and
Margaliot determined the exponent for further factor patterns and obtained
upper bounds by separating the two triangular directions
\cite{ZaraiMargaliot}.

A nonsingular TN matrix is basic oscillatory if it admits an elementary
bidiagonal factorization with exactly one positive factor at each adjacent
index on each triangular side.  For this class,
Fallat and Liu expressed \(e(A)\) as the maximum of the two
triangular exponents \cite[Theorem~12]{FallatLiu}.  They also computed
the individual corner times from a planar normal form and proved, for
each triangular direction, that the times at complementary minor orders
sum to \(n\)
\cite[Lemmas~13 and~15]{FallatLiu}.  Lemma~15 also gives the sum
\(n(n-1)/2\) over all nontrivial minor orders in each direction.
In particular, every basic
oscillatory matrix satisfies \(e(A)\ge\lceil n/2\rceil\)
\cite[Theorem~16]{FallatLiu}.  These corner times retain information
that is lost upon taking their maximum.

Alseidi and Garloff considered the exponent at a prescribed order of minors
\cite{AlseidiGarloff}.  For an oscillatory matrix \(A\), its exponent at minor order \(k\) is
\[
e_k(A)=\min\{m\ge1:\text{all minors of order \(k\) of \(A^m\) are positive}\}.
\]
Then \(e_n(A)=1\) and \(e(A)=\max_{1\le k\le n}e_k(A)\).
We call \((e_1(A),\ldots,e_n(A))\) the \emph{minor-order exponent profile}
of \(A\).  It distinguishes matrices having the same exponent \(e(A)\)
but different values of \(e_k(A)\) at individual minor orders.
The fixed-order corner criterion and complementary-order duality appear in
\cite[Theorems~4.1, 4.2 and~6.2]{XieProfiles}.  Every oscillatory matrix
satisfies \(e_k(A)\le\max\{k,n-k\}\) for \(1\le k<n\) and
\(\lvert e_{k+1}(A)-e_k(A)\rvert\le1\) for \(1\le k<n-1\)
\cite[Theorems~5.4 and~5.6]{XieProfiles}.

Matrix sparsity and inverse sparsity constrain complementary parts of the
profile.  Oscillatory band matrices were studied by Price and Metelmann
\cite{Price,Metelmann}.  Branquinho, Foulqui\'e-Moreno and Ma\~nas
characterized banded total positivity and proved that it is preserved under
multiplication with the bandwidths added \cite{BranquinhoFoulquieManas}.
Classical work also relates total nonnegativity to inverses of
\(M\)-matrices and to tridiagonal inverse structure
\cite{Markham,PenaM,McDonaldEtAl}.  Green matrices provide a concrete
setting for these relations \cite{DelgadoPenaPena,OlshevskyStrangZhlobich}.

Johnson, Olesky and van den Driessche studied the minimum number of
elementary bidiagonal factors for a given matrix, allowing arbitrary nonzero
parameters over a field.  For a nonsingular triangular matrix of order
\(n\), they proved that at most \((n-1)^2\) such factors and one diagonal
factor suffice \cite[Theorem~15]{JohnsonOleskyDriessche}.
For TN matrices and positive parameters, Fomin and Zelevinsky proved
that each triangular index sequence in a shortest elementary
factorization is a shortest adjacent-transposition representation of a
permutation \cite[Theorem~25]{FominZelevinsky}.

These considerations lead to two questions: which profiles occur in the
basic class, and how do simultaneous bandwidth constraints on a matrix
and its inverse restrict its profile at a fixed matrix order?  The latter
question includes the possible bandwidth pairs, the largest exponent and
exponent sum within each pair, and the number of elementary factors needed
for a realization.  A related question is how the positive factor count
controls the loss of exponent sum outside the symmetric basic class.

\section{Preliminaries}
\label{sec:2}

\begin{definition}[Index sets and minors]
Write \([n]=\{1,\ldots,n\}\).  If
\(I=\{i_1<\cdots<i_r\}\) and \(J=\{j_1<\cdots<j_r\}\), then
\(A[I,J]\) is the submatrix in rows \(I\) and columns \(J\), and
\(\Delta_{I,J}(A)=\det A[I,J]\) is the corresponding \(r\)-minor.
The complement of \(I\) in \([n]\) is denoted by \(I^c\).
\end{definition}

\begin{definition}[Entrywise positivity and support]
For a real matrix \(B\), \(B>0\) and \(B\ge0\) mean entrywise strict
positivity and nonnegativity.  Its support is
\(\operatorname{supp}B=\{(i,j):b_{ij}\ne0\}\).
\end{definition}

For equally sized index sets \(I,J\), the Cauchy--Binet formula reads
\[
 \Delta_{I,J}(AB)
 =\sum_{\substack{K\subseteq[n]\\|K|=|I|}}
   \Delta_{I,K}(A)\Delta_{K,J}(B).
\]
If \(A,B\) are TN, every term is nonnegative.  Thus a minor of their
product is positive exactly when at least one intermediate index set
\(K\) makes both minors in the corresponding summand positive.  We repeatedly use this observation
for products of elementary bidiagonal matrices.

We shall use the following classical criterion.

\begin{theorem}[{\cite[Theorem~1]{FallatRemark}}]
\label{thm:1}
Let \(A=(a_{ij})\) be TN.  Then \(A\) is oscillatory if and only if it is
nonsingular and
\[
 a_{i,i+1}>0,\qquad a_{i+1,i}>0,
 \qquad i=1,\ldots,n-1.
\]
\end{theorem}

\begin{lemma}[{\cite[Theorem~1.13]{Pinkus}}]
\label{lem:1}
Every principal minor of a nonsingular TN matrix is positive.
\end{lemma}

For a nonsingular TN matrix \(A\) and integers \(s\ge m\ge1\),
Cauchy--Binet and positivity of principal minors show that
\(\Delta_{I,J}(A^m)>0\) implies \(\Delta_{I,J}(A^s)>0\)
\cite[Lemma~4.3]{XieProfiles}; for oscillatory matrices, see also
\cite[p.~468]{FallatLiu}. Thus the positivity of each fixed minor
persists under further powers.

For \(1\leq k<n\), put
\[
 L_k=\{1,\ldots,k\},
 \qquad
 R_k=\{n-k+1,\ldots,n\}.
\]
\begin{lemma}[{\cite[Theorem~4.1]{XieProfiles}}]
\label{lem:4}
Let \(B\) be an \(n\)-by-\(n\) nonsingular TN matrix and
\(1\leq k<n\).  Then every \(k\)-minor of \(B\) is positive if and only if
\[
 \det B[L_k,R_k]>0
 \qquad\text{and}\qquad
 \det B[R_k,L_k]>0.
\]
\end{lemma}

For an oscillatory matrix \(A\), define the first-positive corner times
\[
 \tau_k^{\mathrm{UR}}(A)
 =
 \min\{m\geq1:\det((A^m)[L_k,R_k])>0\},
\]
\[
 \tau_k^{\mathrm{LL}}(A)
 =
 \min\{m\geq1:\det((A^m)[R_k,L_k])>0\}.
\]

\begin{corollary}[{\cite[Theorem~4.2]{XieProfiles}}]
\label{thm:2}
For every oscillatory matrix \(A\) and \(1\leq k<n\),
\[
 e_k(A)=
 \max\{\tau_k^{\mathrm{UR}}(A),\tau_k^{\mathrm{LL}}(A)\}.
\]
\end{corollary}

Put
\[
 D=\operatorname{diag}(1,-1,1,-1,\ldots),
 \qquad
 A^\vee=DA^{-1}D.
\]

\begin{proposition}[{\cite[Theorem~6.2]{XieProfiles}}]
\label{thm:3}
If \(A\) is oscillatory, then \(A^\vee\) is oscillatory,
\((A^\vee)^\vee=A\), and
\[
 e_k(A^\vee)=e_{n-k}(A),
 \qquad 1\leq k<n.
\]
In particular, \(e(A^\vee)=e(A)\).
\end{proposition}

For a square matrix \(M=(m_{ij})\), define its lower and upper
half-bandwidths by
\[
\operatorname{bw}_{-}M=\max\bigl(\{0\}\cup\{i-j:m_{ij}\ne0\}\bigr),
\quad
\operatorname{bw}_{+}M=\max\bigl(\{0\}\cup\{j-i:m_{ij}\ne0\}\bigr).
\]
Put \(\operatorname{bw}M=\max\{\operatorname{bw}_{-}M,
\operatorname{bw}_{+}M\}\).  We call \(M\) \((p,q)\)-banded when
\(\operatorname{bw}_{-}M\le p\) and \(\operatorname{bw}_{+}M\le q\).
Throughout, bandwidth refers to half-bandwidth.  The notation
\((p,q)\)-banded specifies upper bounds; exact bandwidths are stated
as equalities.

\section{Basic matrices and exponent sums}
\label{sec:3}

\subsection{Elementary factors and corner times}

For \(1\le i<n\) and \(a>0\), put \(x_i(a)=I+aE_{i,i+1}\), where
\(E_{ij}\) has a single entry one in position \((i,j)\).
Right multiplication by \(x_i(a)\) adds \(a\) times column \(i\) to
column \(i+1\).  For a finite
sequence \(w=(i_1,\ldots,i_L)\), set
\[
U(w)=x_{i_1}(a_1)\cdots x_{i_L}(a_L),\qquad a_r>0.
\]
Different occurrences may have different positive parameters.  Let
\(\ell_\uparrow(w)\) be the largest length of a sequence
\(i,i+1,\ldots,j-1\) occurring as a subsequence of \(w\).  Define
\(\ell_\downarrow(w)\) using consecutive decreasing indices.  The
selected positions need not be consecutive in \(w\).
For example, \(w=(3,1,2,1)\) has \(\ell_\uparrow(w)=2\) and
\(\ell_\downarrow(w)=3\), witnessed by the subsequences \((1,2)\)
and \((3,2,1)\), respectively.

\begin{lemma}
\label{lem:5}
Suppose that \(w_-,w_+\) each contain every index \(1,\ldots,n-1\), and
let \(\Delta\) be positive diagonal.  Then
\[
A=U(w_-)^T\Delta U(w_+)
\]
is oscillatory, and
\[
\operatorname{bw}_{\pm}A=\ell_\uparrow(w_\pm),
\qquad
\operatorname{bw}_{\pm}A^{-1}=\ell_\downarrow(w_\pm).
\]
When \(w_-=w_+=w\) with the same parameters, \(A\) is symmetric positive
definite.  When all parameters and all diagonal entries of \(\Delta\)
equal one, \(A\) is an integer matrix of determinant one.
\end{lemma}

\begin{proof}
Expanding the product of the elementary matrices gives, for \(i<j\),
\begin{equation}
\label{eq:1}
U(w)_{ij}>0
\quad\Longleftrightarrow\quad
(i,i+1,\ldots,j-1)\text{ occurs as a subsequence of }w.
\end{equation}
Indeed, \(E_{ab}E_{cd}=\delta_{bc}E_{ad}\), and all nonzero terms are
positive products of parameters.

Write \(U_\pm=U(w_\pm)\).  For \(i<j\),
\[
A_{ij}=\sum_{r\le i}(U_-)_{ri}\Delta_{rr}(U_+)_{rj}.
\]
If \((U_+)_{ij}>0\), the term \(r=i\) is positive.  Conversely, a
positive term implies \((U_+)_{rj}>0\); the subsequence in
\eqref{eq:1} then contains the suffix \(i,\ldots,j-1\).
Thus \(A\) and \(U_+\) have the same upper support.  Transposition gives
the lower support.

Set \(V_\pm=DU_\pm^{-1}D\).  Since
\(Dx_i(a)^{-1}D=x_i(a)\), each \(V_\pm\) is the elementary product in
reverse order, and
\[
A^\vee=V_+\Delta^{-1}V_-^T.
\]
For \(i<j\), its \((i,j)\) entry is
\(\sum_{r\ge j}(V_+)_{ir}\Delta_{rr}^{-1}(V_-)_{jr}\).
The term \(r=j\) and the prefix \(i,\ldots,j-1\) in
\eqref{eq:1} show that its upper support is that of \(V_+\).
Transposition gives the lower support.  Reversal exchanges increasing and
decreasing subsequences, proving the bandwidth formulas.

Elementary bidiagonal matrices are TN, so Cauchy--Binet makes \(A\) TN.
Its determinant is \(\det\Delta>0\), and the occurrence of every index
makes both first off-diagonals positive.  The oscillatory criterion applies.
The assertions about symmetry and integer realizations follow from the
displayed factorization.
\end{proof}

To describe minors of these products, represent a \(k\)-subset of
\([n]\) by a string of \(n\) zeros and ones: position \(j\) contains
one exactly when \(j\) belongs to the subset.  Thus \(L_k\) and \(R_k\)
are represented by \(1^k0^{n-k}\) and \(0^{n-k}1^k\), respectively;
an exponent on a symbol denotes the number of repetitions.

An operation at index \(i\) replaces \(10\) in positions \(i,i+1\)
by \(01\), leaving the string unchanged in all other cases.  For a list
\(Q=(q_1,\ldots,q_\ell)\) with \(q_r\in\{1,\ldots,n-1\}\), one
\emph{pass} applies these operations in the listed order, using the
updated string at each step.  Let \(t_k(Q)\) be the least positive number
of passes taking \(1^k0^{n-k}\) to \(0^{n-k}1^k\), with value infinity
if this never happens.  For example, with \(n=5\), \(k=2\) and
\(Q=(1,2,3,4)\), successive passes give
\(11000\longrightarrow10001\longrightarrow00011\), so \(t_2(Q)=2\).

The separation of the two triangular directions is given in
\cite[Proposition~8]{ZaraiMargaliot}; see also
\cite[Lemma~5.1]{XieProfiles} for factors grouped in
lower--diagonal--upper order.
The next lemma records the corner-time identities in subset notation,
together with the coordinate comparison used in the later proofs.

\begin{lemma}
\label{lem:6}
Let \(U_Q=x_{q_1}(a_1)\cdots x_{q_\ell}(a_\ell)\), where all
\(a_r>0\).  The minor \(\Delta_{X,Y}(U_Q)\) is positive if and only
if the string for \(X\) can be taken to the string for \(Y\) by the
listed operations, each of which may either be performed or omitted.
Performing every available exchange puts each ordered element of the
resulting subset at least as far to the right as any such choice of
omissions.  Consequently, for \(1\le k<n\) and every integer \(m\ge1\),
\[
 \Delta_{L_k,R_k}(U_Q^m)>0
 \quad\Longleftrightarrow\quad t_k(Q)\le m.
\]
If \(A\) is an oscillatory product of positive elementary upper and lower
factors and positive diagonal matrices, let \(Q_+\) list its upper
factor indices in product order.  Let \(Q_-\) list the upper factor
indices in the transposed product \(A^T\), with the order of all
factors reversed on transposing.  Then
\[
 \tau_k^{\mathrm{UR}}(A)=t_k(Q_+),
 \qquad
 \tau_k^{\mathrm{LL}}(A)=t_k(Q_-).
\]
\end{lemma}

\begin{proof}
For one factor \(x_i(a)\), the positive minors with a fixed row set
\(X\) have column set either \(X\) itself or the set obtained by
replacing \(i\) by \(i+1\), when \(i\in X\) and \(i+1\notin X\).
Their values are one and \(a\), respectively.  Repeated Cauchy--Binet
therefore expresses a minor of \(U_Q\) as a sum of nonnegative terms
indexed by the permitted choices of exchanges.  A term is positive
exactly when every step is permitted.  This proves the first assertion.

We next compare two subsets with ordered elements
\(x_1<\cdots<x_k\) and \(y_1<\cdots<y_k\), where \(x_j\le y_j\)
for every \(j\).  Applying the operation at \(i\) to both subsets
preserves these inequalities.  The only possible failure would require
\(x_j=y_j=i\), with \(x_j\) moving right and \(y_j\) staying put.
The latter can happen only if \(j<k\) and \(y_{j+1}=i+1\).  But then
\(i<x_{j+1}\le y_{j+1}=i+1\), so the first subset is blocked as well.
Thus failure is impossible.

Each operation can only increase an ordered element.  Induction over
the list now shows that performing every available exchange gives
ordered elements at least as large as any sequence with omissions.
The set \(R_k\) has the largest possible ordered elements,
\(n-k+1,\ldots,n\).  Hence it can be reached from \(L_k\) if and
only if performing every available exchange reaches it.  Once reached,
its string \(0^{n-k}1^k\) is unchanged by further operations.  Applying
the same reasoning to \(m\) repetitions of \(Q\) proves the equivalence.

For the full product \(A^m\), Cauchy--Binet also permits a leftward
exchange \(01\mapsto10\) at a lower factor; a diagonal factor leaves
the index set unchanged.  Compare any permitted sequence with the one
that performs every available rightward exchange and leaves the subset
unchanged at all lower and diagonal factors.  The coordinate comparison
just proved is preserved at upper factors.  At lower factors, the first
sequence can only move left while the second stays fixed.  It is
therefore preserved throughout the product.  Since the second sequence
is itself permitted, \(R_k\) is reached from \(L_k\) exactly when the
rightward operations listed in \(Q_+\), repeated \(m\) times, reach it.
This proves the upper-right formula.  Transposing the product turns the
lower-left minor into an upper-right minor and gives the formula for
\(Q_-\).
\end{proof}

Let \(c=(c_1,\ldots,c_{n-1})\) be a permutation of
\(1,\ldots,n-1\), attach a positive parameter \(a_i\) to index \(i\), and
set
\[
U(c)=x_{c_1}(a_{c_1})\cdots x_{c_{n-1}}(a_{c_{n-1}}).
\]
Write \(\pi_i\) for the position of index \(i\) in \(c\), and define its
direction string \(d_1,\ldots,d_{n-2}\) by recording whether factor
\(i\) occurs before or after factor \(i+1\):
\[
d_i=
\begin{cases}
0,&\pi_i<\pi_{i+1},\\
1,&\pi_i>\pi_{i+1}.
\end{cases}
\]
For example, \(c=(2,1,3,4)\) has direction string \((1,0,0)\).
Every binary direction string occurs: prescribe the indicated order on
each pair of adjacent indices and choose any ordering consistent with it.
The underlying graph is a path, so these prescriptions have no directed
cycle.

Let \(r_0(c)\) and \(r_1(c)\) be the longest run lengths of zeros and ones,
respectively, in the direction string, where a run is a consecutive block
of equal symbols; a missing run has length zero.  Put
\[
N_1[a,b]=\#\{i:a\le i\le b,\ d_i=1\},
\qquad
N_0[a,b]=\#\{i:a\le i\le b,\ d_i=0\},
\]
with both counts zero when \(a>b\).

The bidiagonal factorization theorem gives every nonsingular TN matrix
a factorization \(A=L\Delta U\), where \(L,U\) are unit lower and
upper triangular TN matrices and \(\Delta\) is positive diagonal
\cite[Theorem~1]{ZaraiMargaliot}. This factorization is unique.
If \(A=A^T\), uniqueness gives \(L=U^T\).

A nonsingular TN matrix is \emph{basic oscillatory} when its standard
successive elementary bidiagonal factorization has exactly one positive
parameter at each adjacent index on each triangular side
\cite{FallatFiedlerMarkham}.  Its triangular factors admit a normal
form obtained by commuting elementary matrices with nonadjacent indices
\cite[Theorem~2.2]{FallatFiedlerMarkham}; see also the planar form in
\cite[Lemma~11]{FallatLiu}.  The next lemma puts this form in terms of
the products \(U(c)\) and identifies the symmetric members.

\begin{lemma}
\label{lem:7}
The basic oscillatory matrices are exactly
\[
A=U(c_-)^T\Delta U(c_+),
\]
where \(c_-,c_+\) are permutations of \(1,\ldots,n-1\) and \(\Delta\)
is positive diagonal.  The symmetric members are exactly
\(U(c)^T\Delta U(c)\), with the same parameters in the two factors.
\end{lemma}

\begin{proof}
In the standard bidiagonal factorization \cite[Theorem~1]{ZaraiMargaliot},
the upper factors occur in the order
\[
(n-1);\ (n-2,n-1);\ \ldots;\ (1,2,\ldots,n-1).
\]
The lower factors have the transposed form, and zero parameters give
identity factors.  Retaining one positive factor of each index on each
side therefore gives the displayed representation of a basic matrix.

Conversely, start with any permutation \(c\).  Split \(1,\ldots,n-1\) into
maximal intervals whose adjacent direction bits are zero.  List each
interval increasingly and list the intervals in decreasing order of their
left endpoints.  The resulting permutation \(c'\) has the same relative
order as \(c\) for every pair of numerically adjacent indices.  Move
the first index of \(c'\) to the front of \(c\), then its second index
to the second position, and continue.  Each index passed during this
process differs from the moved index by more than one: otherwise the
two lists would disagree on the order of that adjacent pair.
The corresponding elementary factors commute, so these exchanges do
not change \(U(c)\).  Each increasing interval \([a,b]\) of \(c'\) is
a subsequence of the segment \((a,a+1,\ldots,n-1)\) in the fixed order
above.  Set the unused parameters to zero.
This gives the required basic factorization.  Oscillation follows from
Lemma~\ref{lem:5}.

The unit lower--diagonal--unit upper factorization is unique.  Comparing it
with its transpose for a symmetric matrix gives \(U(c_-)=U(c_+)\), with
the same positive diagonal factor.  This proves the symmetric assertion.
\end{proof}

The corner times computed in \cite[Lemma~13]{FallatLiu} have the
following explicit expression in the direction string.  We prove it
directly to obtain the full profile in the notation used here.

\begin{theorem}
	\label{thm:4}
	Let \(\Delta\) be positive diagonal and set
	\[
	A(c)=U(c)^T\Delta U(c).
	\]
	Then \(A(c)\) is symmetric positive definite and oscillatory, and
	\[
	\operatorname{bw}A(c)=1+r_0(c),
	\qquad
	\operatorname{bw}A(c)^{-1}=1+r_1(c).
	\]
	For \(1\le k\le n-1\),
	\begin{equation}
	\label{eq:2}
	e_k(A(c))=
	\begin{cases}
	k+N_1[k,n-k-1],&2k\le n,\\[3pt]
	n-k+N_0[n-k,k-1],&2k\ge n.
	\end{cases}
	\end{equation}
	At \(2k=n\), both expressions equal \(k\); as usual,
	\(e_n(A(c))=1\).
\end{theorem}

\begin{proof}
	Lemma~\ref{lem:5} gives positive definiteness and oscillation, and
	expresses the two bandwidths as
	\(\ell_\uparrow(c)\) and \(\ell_\downarrow(c)\).  Each index occurs
	once, so a consecutive increasing subsequence corresponds to a run of
	zeros in the direction string, and a decreasing one to a run of ones.
	Their maximum lengths are \(1+r_0(c)\) and \(1+r_1(c)\).

	It remains to compute the profile.  Repeat \(c\) periodically as a
	sequence of adjacent moves.  Lemma~\ref{lem:6} shows
	that both remote
	order-\(k\) corner thresholds of \(A(c)\) equal the number \(t_k(c)\) of
	passes carrying \(1^k0^{n-k}\) to \(0^{n-k}1^k\).  Thus the exact
	two-corner formula gives
	\begin{equation}
	\label{eq:3}
	e_k(A(c))=t_k(c).
	\end{equation}

	Put \(h=n-k\).  Label the ones in the initial string from right to
	left by \(p=1,\ldots,k\), and the zeros from left to right by
	\(b=1,\ldots,h\).  An exchange preserves the relative order of the
	ones and of the zeros.  To reach the final string, each labeled one
	must therefore cross each labeled zero exactly once.
	The \(p\)-th one starts at position \(k-p+1\).  Before meeting the
	\(b\)-th zero it has crossed precisely \(b-1\) zeros, so their
	exchange occurs at index
	\[
	r_{p,b}=k-p+b.
	\]
	Write \(s_{p,b}\) for the pass containing this exchange.  The
	\(b\)-th zero must first cross the preceding one, giving the exchange
	\((p-1,b)\) when \(p>1\).  The \(p\)-th one must first cross the
	preceding zero, giving \((p,b-1)\) when \(b>1\).  Once these
	exchanges have occurred, no symbol remains between the pair.
	Each index occurs once per pass, so an exchange waits for the next
	pass precisely when a required predecessor uses a later index in
	the factor order.

	For \(r=k-p+b\), the two resulting lower bounds on its pass number are
	\[
	s_{p-1,b}+1-d_r\quad(p>1),
	\qquad
	s_{p,b-1}+d_{r-1}\quad(b>1).
	\]
	Indeed, the predecessor indices are \(r+1\) and \(r-1\), and
	\(1-d_r\) and \(d_{r-1}\) record whether they occur later than \(r\).
	Its pass number is the maximum of the available bounds.
	The first pair is initially adjacent, so \(s_{1,1}=1\).

	Define \(F(r)=N_1[1,r-1]\).  We claim that
	\[
	s_{p,b}=p+F(k-p+b)-F(k).
	\]
	The initial value is correct.  For the induction step, write
	\(r=k-p+b\).  If \(p>1\), the first predecessor contributes
	\[
	s_{p-1,b}+1-d_r
	=(p-1)+F(r+1)-F(k)+1-d_r
	=p+F(r)-F(k).
	\]
	If \(b>1\), the second contributes
	\[
	s_{p,b-1}+d_{r-1}
	=p+F(r-1)-F(k)+d_{r-1}
	=p+F(r)-F(k).
	\]
	Thus all existing predecessors give the same value, proving the claim
	by induction on \(p+b\).

	The exchange \((k,h)\) is the last one, so the final pass is
	\[
	t_k(c)=s_{k,h}=k+F(h)-F(k).
	\]
	For \(k\le h\), this is \(k+N_1[k,h-1]\).  For \(k\ge h\), it is
	\(k-N_1[h,k-1]=h+N_0[h,k-1]\).
	Together with \eqref{eq:3}, this proves
	\eqref{eq:2}.
\end{proof}

The complementary identity in the next lemma was proved in
\cite[Lemma~15]{FallatLiu}.  The remaining inequalities follow from
the explicit corner formula.

\begin{lemma}
\label{lem:8}
For every factor order \(c\), its one-sided profile \(t_k(c)\) satisfies
\[
 t_k(c)+t_{n-k}(c)=n,
 \qquad
 t_k(c)\leq\max\{k,n-k\}\quad(1\le k<n),
\]
and
\[
 |t_{k+1}(c)-t_k(c)|\leq1\quad(1\le k<n-1).
\]
\end{lemma}

\begin{proof}
The complementary identity follows because the same central interval
occurs at orders \(k\) and \(n-k\), with zeros and ones exchanged.  If
\(2k\leq n\), then
\[
 k\leq t_k(c)=k+N_1[k,n-k-1]\leq k+(n-2k)=n-k;
\]
the other half follows by complementation.

For every \(1\le k<n-1\), the final formula in the preceding proof gives
\[
 t_{k+1}(c)-t_k(c)=1-d_k-d_{n-k-1}\in\{-1,0,1\}.
\]
This proves the Lipschitz bound.
\end{proof}

\subsection{Exponent sums and factor counts}

For a fixed unit upper triangular TN matrix \(U\), denote by
\(\ell(U)\) the least number of factors in a representation
\[
 U=x_{q_1}(a_1)\cdots x_{q_s}(a_s),\qquad a_j>0.
\]
Only the elementary upper factors are counted.  For a permutation
\(z=(z_1,\ldots,z_n)\), write
\(\operatorname{inv}(z)=\#\{(i,j):i<j,\ z_i>z_j\}\).
We use the following upper triangular form of the positive factorization
theorem.

\begin{lemma}[{\cite[Theorems~12, 25 and~26]{FominZelevinsky}}]
\label{lem:12}
Every unit upper triangular TN matrix is a product of positive
elementary upper factors, with the empty product allowed.
Let \(U=x_{q_1}(a_1)\cdots x_{q_s}(a_s)\), with \(a_j>0\).
Start with \((1,\ldots,n)\) and, successively at
positions \(q_j,q_j+1\), exchange the two entries when the left one
is smaller.  This representation has \(s=\ell(U)\) if and only if
every one of the \(s\) exchanges occurs.  In that case the resulting
permutation has \(s\) inversions.
\end{lemma}

The source permits positive lower and diagonal factors as well.  A
positive lower factor would leave a positive entry below the diagonal
in the product, so none can occur in a representation of \(U\).
Diagonal factors can be moved past upper factors by changing their
positive parameters; the unit diagonal of \(U\) leaves no diagonal
factor.  Thus the minimum in that theorem is precisely \(\ell(U)\).

\begin{samepage}
\begin{lemma}
\label{lem:13}
Let \(Q\) contain every index \(1,\ldots,n-1\). Starting from
\((1,\ldots,n)\), repeat the passes of adjacent exchanges described
in Lemma~\ref{lem:12}. Let \(a_r\) be the inversion number after
\(r\) passes, with \(a_0=0\). Then
\begin{equation}
\label{eq:6}
 \sum_{k=1}^{n-1}t_k(Q)+a_1\le\binom n2+n-1.
\end{equation}
Equality holds if and only if
\[
 a_2-a_1=\#\{k:1\le k<n,\ t_k(Q)>1\}.
\]
Moreover, the increments \(a_{r+1}-a_r\) are nonincreasing for
\(r\ge0\).
\end{lemma}
\end{samepage}

\begin{proof}
On permutations, let \(\mathcal R_i\) exchange the entries in
positions \(i,i+1\) if the left entry is smaller. Let
\(\mathcal L_j\) exchange the values \(j,j+1\) if \(j\) occurs
to the left of \(j+1\). In all other cases the respective operation
does nothing. Each actual exchange of either kind increases the
inversion number by one. For \(\mathcal L_j\), every third value
is either smaller than both \(j,j+1\) or larger than both, so its
total contribution to the inversion number is unchanged.

These two kinds of operations commute:
\(\mathcal R_i\mathcal L_j=\mathcal L_j\mathcal R_i\).
If positions \(i,i+1\) contain exactly \(j,j+1\), either order makes
the same single exchange or leaves both entries fixed. Otherwise,
an exchange of adjacent positions does not reverse the order of the
values \(j,j+1\), and interchanging consecutive values does not change
their comparisons with a third value. Thus neither operation changes
whether the other occurs. The position exchange and the relabeling
then commute.

Let \(F\) apply the operations \(\mathcal R_i\) in the order \(Q\),
and let \(G\) apply the operations \(\mathcal L_j\) in the reverse
order. Put \(\varepsilon=(1,\ldots,n)\). The identities
\(\mathcal R_i\varepsilon=\mathcal L_i\varepsilon\), together with
commutation, give \(F\varepsilon=G\varepsilon\): replace the first
operation on \(\varepsilon\) and commute it past the remaining
operations, then repeat. Also \(FG=GF\). Consequently, if
\(z^{(r)}=F^r\varepsilon\), then
\[
 Gz^{(r)}=GF^r\varepsilon=F^rG\varepsilon
        =F^{r+1}\varepsilon=z^{(r+1)}.
\]

Fix \(k\), and mark a position by one when its value is at most
\(k\), and by zero otherwise. Each \(\mathcal R_i\) induces
exactly \(10\mapsto01\) on this string. Thus the values
\(1,\ldots,k\) first occupy the last \(k\) positions after \(t_k(Q)\)
passes. An unfinished string must change in the next complete pass:
otherwise no rightward exchange could occur, and an adjacent \(10\)
would persist until its index occurs in \(Q\), when it must move.

Let \(m_{r,k}\) count the actual \(\mathcal L_k\) exchanges in the
\(r\)th pass of \(G\). All other value exchanges leave the marked
string unchanged. Hence \(m_{r,k}>0\) when \(r\le t_k(Q)\).
After the two groups of values have separated, exchanges within either
group preserve that separation, and \(\mathcal L_k\) cannot act.
Therefore
\[
 m_{r,k}>0\quad\Longleftrightarrow\quad r\le t_k(Q).
\]
Counting all exchanges gives
\begin{equation}
\label{eq:7}
 a_{r+1}-a_r=\sum_{k=1}^{n-1}m_{r+1,k}
 \ge\#\{k:1\le k<n,\ t_k(Q)>r\}.
\end{equation}

An occurrence in the fixed list \(G\), once ineffective, remains
ineffective in later passes. To see this, let \(H\) be the prefix
before that occurrence, whose label is \(k\), and put
\(y=Hz^{(r-1)}\). The corresponding state in the next pass is
\(Hz^{(r)}=HFz^{(r-1)}=Fy\), since each value operation commutes
with \(F\). If \(\mathcal L_k y=y\), then
\(\mathcal L_k Fy=F\mathcal L_k y=Fy\). Consequently
\(m_{r+1,k}\le m_{r,k}\), which proves the assertion about increments.
It also shows that
\[
 d_r=(a_{r+1}-a_r)-\#\{k:1\le k<n,\ t_k(Q)>r\}
     =\sum_{k=1}^{n-1}\max\{m_{r+1,k}-1,0\}
\]
satisfies \(d_0\ge d_1\ge d_2\ge\cdots\ge0\).

The permutations eventually reach \((n,\ldots,1)\): a permutation
left fixed by a complete pass must be decreasing, and every other pass
increases the inversion number. Thus \(a_r=\binom n2\) for large
\(r\). Summing \eqref{eq:7}, with its excess \(d_r\), over \(r\ge1\)
gives
\[
 \sum_{k=1}^{n-1}t_k(Q)+a_1+\sum_{r\ge1}d_r=\binom n2+n-1.
\]
This proves \eqref{eq:6}. Since the \(d_r\) are nonnegative and
nonincreasing, equality holds exactly when \(d_1=0\), which is the
stated condition.
\end{proof}

For symmetric basic matrices, the sum identity follows from the
one-sided identities in \cite[Lemma~15]{FallatLiu}.  The next theorem
bounds the sum for every symmetric oscillatory matrix in terms of its
minimum factor count.

\begin{theorem}
\label{thm:12}
Let \(A=U^T\Delta U\) be a symmetric oscillatory matrix of order
\(n\), where \(U\) is unit upper triangular and \(\Delta\) is
positive diagonal.  There is a symmetric basic oscillatory matrix
\(B\) of the same order such that \(e_k(A)\le e_k(B)\) for every
\(1\le k<n\).  Moreover,
\[
 e_k(A)+e_{n-k}(A)\le n\qquad(1\le k<n),
\]
and
\begin{equation}
\label{eq:8}
 \sum_{k=1}^{n-1}e_k(A)+\ell(U)\le\binom n2+n-1.
\end{equation}
Equality in \eqref{eq:8} holds if and only if
\begin{equation}
\label{eq:10}
 \ell(U^2)-\ell(U)
 =\#\{k:1\le k<n,\ \Delta_{L_k,R_k}(A)=0\}.
\end{equation}
In particular, \(\sum_{k=1}^{n-1}e_k(A)\le\binom n2\).
Equality in this last bound holds if and only if \(A\) is basic.
It is also equivalent to equality in every complementary-order bound.
\end{theorem}

\begin{proof}
The bidiagonal factorization and symmetry give the stated form of
\(A\).  Choose a shortest positive elementary representation of \(U\),
with index sequence \(Q\).  Every adjacent index occurs in \(Q\), since
otherwise a first off-diagonal entry of \(A\) would vanish.  Select one occurrence of each
index, preserving its position, and set all other parameters to zero.
The resulting matrix \(B=V^T\Delta V\) is symmetric basic by
Lemma~\ref{lem:7}.

Repeated Cauchy--Binet expresses every minor of \(A^r\) as a polynomial
with nonnegative coefficients in the factor parameters.  Setting some
parameters to zero therefore gives
\[
 0\le\Delta_{I,J}(B^r)\le\Delta_{I,J}(A^r)
 \qquad(r\ge1).
\]
Thus \(e_k(A)\le e_k(B)\).  The complementary identity for \(B\)
in Lemma~\ref{lem:8} gives \(e_k(A)+e_{n-k}(A)\le n\).

The first-pass inversion number of \(Q\) is
\(a_1=\ell(U)\), by Lemma~\ref{lem:12}. Symmetry,
Lemma~\ref{lem:6} and Corollary~\ref{thm:2} give
\(e_k(A)=t_k(Q)\). Applying Lemma~\ref{lem:13} proves \eqref{eq:8}.

In fact, \(a_r=\ell(U^r)\) for every \(r\ge1\). For each \(k\),
the positions occupied by the values \(1,\ldots,k\) after \(r\)
permutation passes form the componentwise greatest subset \(J\)
with \(\Delta_{L_k,J}(U^r)>0\), by Lemma~\ref{lem:6}.
These subsets are determined by \(U^r\); their successive set
differences determine the position of every value. Thus every positive
factorization of \(U^r\) has the same final permutation after one
pass. In a shortest factorization all exchanges occur, by
Lemma~\ref{lem:12}, so its length is \(a_r\).
Also \(t_k(Q)>1\) exactly when \(\Delta_{L_k,R_k}(A)=0\).
The equality condition in Lemma~\ref{lem:13} is therefore
\eqref{eq:10}.

Since every index occurs, \(\ell(U)\ge n-1\).  Equality holds
exactly when \(A\) is basic, by Lemma~\ref{lem:7}.  Therefore
\eqref{eq:8} gives a strict sum bound when \(A\) is not basic.
For basic \(A\), Lemma~\ref{lem:8} gives the sum \(\binom n2\).
Finally, since every complementary sum is at most \(n\), equality
in the total sum is equivalent to equality in all these bounds.
\end{proof}

The same proof and Lemma~\ref{lem:13} show that the sequence
\(\ell(U^r)\) has nonincreasing increments:
\[
 \ell(U^{r+1})-\ell(U^r)
 \le \ell(U^r)-\ell(U^{r-1})\qquad(r\ge1),
\]
where \(\ell(I)=0\). Lemma~\ref{lem:12} gives
\(n-1\le\ell(U)\le\binom n2\).

\begin{corollary}
\label{cor:7}
For every integer \(L\) with \(n-1\le L\le\binom n2\), the largest
value of \(\sum_{k=1}^{n-1}e_k(A)\) among symmetric oscillatory
matrices \(A=U^T\Delta U\) with \(\ell(U)=L\) is
\[
 \binom n2+n-1-L.
\]
The maximum has a symmetric integer realization of determinant one.
\end{corollary}

\begin{proof}
Choose \(h,s\) so that
\[
 L=n-1+\binom h2+s,\qquad 1\le h\le n-1,\quad 0\le s<h,
\]
where \(s=0\) if \(h=n-1\).  Every stated \(L\) has this form.
Start with \((1,2,\ldots,n-1)\), append the lists
\((1),(2,1),\ldots,(h-1,h-2,\ldots,1)\), and then append
\((h,h-1,\ldots,h-s+1)\).  Lists of length zero are omitted.
Call the resulting sequence \(Q\), set all factor parameters to one,
and put \(A=U(Q)^TU(Q)\).  Lemma~\ref{lem:5} gives a symmetric
oscillatory integer matrix of determinant one.

All exchanges in the first pass on \((1,\ldots,n)\) occur.
The initial list gives \((2,3,\ldots,n,1)\); the succeeding
lists reverse its first \(h\) entries by successive insertion, and
the last list moves the larger entry \(h+2\) left by \(s\) positions
when \(s>0\).  Hence \(\ell(U(Q))=L\) by Lemma~\ref{lem:12}.

Let $a_r$ be the inversion number after $r$ passes, and let $z^{(1)}$
be the permutation after the first pass. The values
$1,2,\ldots,h+1$ occur in decreasing relative order: the initial list
moves $1$ to the last position, the insertion lists reverse
$2,\ldots,h+1$, and the last list, when nonempty, only inserts the
larger value $h+2$. Thus every value exchange $\mathcal L_j$ with
$1\le j\le h$, as defined in the proof of Lemma~\ref{lem:13},
leaves $z^{(1)}$ fixed.

The reverse-list value pass from that proof takes $z^{(1)}$ to the
second-pass permutation. It starts with the reversals of the appended
lists. All their labels are at most $h$, so none acts. The remaining
list is $(n-1,\ldots,1)$, in which each label occurs once.
Consequently, each $\mathcal L_k$ acts at most once in the second pass.
The proof of Lemma~\ref{lem:13} shows that it acts precisely when
$t_k(Q)>1$. Hence
\[
 a_2-a_1=\#\{k:1\le k<n,\ t_k(Q)>1\}.
\]
The equality criterion in that lemma, together with $a_1=L$ and
$e_k(A)=t_k(Q)$, gives
\[
 \sum_{k=1}^{n-1}e_k(A)=\binom n2+n-1-L.
\]
This attains the upper bound in Theorem~\ref{thm:12}.
\end{proof}

\subsection{Factor increments and total deficit}

The increments of the minimum factor count have an exact expression
once the power is sufficiently large. We use the convention
$\ell(I)=0$ from the preceding subsection.

\begin{theorem}
\label{thm:cutoff}
Let \(A=U^T\Delta U\) be a symmetric oscillatory matrix of order
\(n\ge2\), where \(U\) is unit upper triangular and \(\Delta\) is
positive diagonal. Then, for every integer
\(r\ge\lfloor(n-1)/2\rfloor\),
\begin{equation}
\label{eq:cutoff}
 \ell(U^{r+1})-\ell(U^r)
 =\#\{k:1\le k<n,\ e_k(A)>r\}.
\end{equation}
\end{theorem}

We first prove two facts about the adjacent exchanges in
Lemma~\ref{lem:13}. Fix a list \(Q\) containing every index
\(1,\ldots,n-1\), and repeat it on the increasing permutation.
Put \(P=|Q|\), and number the operations by the positive integers,
so that pass \(s\) consists of the times \(((s-1)P,sP]\).
An occurrence means a specified position in \(Q\); different
occurrences may have the same adjacent index. We follow each value
through the permutations and say that it moves right or left according
to its change of position at an exchange.

Every fixed occurrence, once ineffective, remains ineffective in later
passes. This follows from the commutation argument in
Lemma~\ref{lem:13}, with the position and value operations interchanged.
Indeed, retain the notation \(F,G,z^{(r)}\) from that proof and let
\(H\) be the prefix of position operations preceding the occurrence
\(\mathcal R_i\). Since \(z^{(r)}=Gz^{(r-1)}\) and all position
operations commute with \(G\), the states before this occurrence in
successive passes satisfy
\[
 H z^{(r)}=G H z^{(r-1)}.
\]
If \(\mathcal R_i\) fixes \(H z^{(r-1)}\), it therefore fixes
\(H z^{(r)}\). In particular, a value which makes no exchange in a
whole pass never moves again: every occurrence at an index adjacent
to its position is then permanently ineffective.
We also use the facts that each pair of values exchanges at most once
and that the permutations eventually become decreasing.

We shall use the following time comparison. Suppose that a value
\(x\) moves only right from its exchange at index \(q\) to its
exchange at index \(j\ge q\), with actual times
\(\tau_q<\cdots<\tau_j\). Suppose another value has rightward
exchanges at these indices at times
\(\sigma_q<\cdots<\sigma_j\), with
\(\sigma_q\ge\tau_q+P\). These may be its first rightward
crossings of the successive indices after a specified time, even
if it also has leftward exchanges. Then
\begin{equation}
\label{eq:period-separation-propagation}
 \sigma_k\ge\tau_k+P\qquad(q\le k\le j).
\end{equation}
If \(k>q\) is the first index where this fails, then
\(\tau_{k-1}<\sigma_k-P<\tau_k\). The occurrence at index
\(k\) is ineffective at \(\sigma_k-P\), while \(x\) waits at
position \(k\), but effective at \(\sigma_k\), a contradiction.
The same comparison can start at any intermediate index where the
initial separation has been established.

\begin{lemma}
\label{lem:single-exchange}
Suppose that a value makes exactly one exchange in some pass.
After that exchange its direction of motion changes at most once.
If it changes, its last exchange in the former direction and its
first exchange in the new direction occur at the same index and
are exactly \(P\) operations apart. Consequently, no later pass
contains exchanges of this value in both directions.
\end{lemma}

\begin{proof}
By reflecting positions and reversing the order of the values, it
suffices to consider a sole rightward exchange, at index \(p\) and
time \(\tau_p\). The value, say \(b\), stays at position \(p\)
before this exchange and at position \(p+1\) afterward, until the
end of that pass. Every other occurrence of index \(p\) is
therefore ineffective in that pass and remains so. Only the
occurrence whose phase is \(\tau_p\) modulo \(P\) can act there
in the future.

Suppose that \(b\) subsequently moves left, and let \(u\) be the
time of its first such exchange, at index \(j\ge p\). Until \(u\), it moves
only right. Write
\[
 \tau_p<\tau_{p+1}<\cdots<\tau_j<u
\]
for its actual rightward exchange times at the intervening indices.
Let \(c<b\) exchange with \(b\) at time \(u\). This pair has
not exchanged previously, so at the end of the sole-exchange pass
\(c\) is to the left of \(b\), at a position at most \(p\).
For \(p\le k\le j\), let \(\sigma_k\) be its first rightward
exchange at index \(k\) after that pass. These times exist and are
strictly increasing. Since \(c\) stays to the left of \(b\) until
\(u\), and \(b\) is never to the right of position \(j+1\) before
\(u\), its first rightward crossing of \(j\) is the exchange at
\(u\). Thus \(\sigma_j=u\). The only possible phase
at index \(p\) gives \(\sigma_p\ge\tau_p+P\).

Applying \eqref{eq:period-separation-propagation} to these times gives,
in particular,
\begin{equation}
\label{eq:single-exchange-delay}
 u\ge\tau_j+P.
\end{equation}

If \(j=n-1\), then \(b\) has already crossed every larger value.
If \(j<n-1\), it stays at position \(j+1\) throughout
\((\tau_j,u)\). By \eqref{eq:single-exchange-delay}, this interval
contains every occurrence of index \(j+1\): its length is at least
\(P\), and neither endpoint is an operation at that index.
All these occurrences are ineffective and hence permanently so.
The eventual decreasing permutation shows that all values to the
right of this cut are smaller than \(b\). Every larger value is
therefore already to its left and has crossed it. In either case,
\(b\) can never move right again, because no pair exchanges twice.

Finally, any two consecutive exchanges of one value are at most
\(P\) operations apart. Otherwise the occurrence producing the
second exchange was ineffective one period earlier, while the value
was still at the same position. Applying this observation to
\(\tau_j,u\) and using \eqref{eq:single-exchange-delay} gives
\(u=\tau_j+P\). Both exchanges use index \(j\), so they are
successive uses of the same occurrence. They lie in different
passes, and there can be no later change of direction. This proves
the final assertion as well.
\end{proof}

\begin{lemma}
\label{lem:period-separation}
Suppose that each of two values makes exactly one exchange in some
strictly earlier pass. If both move right in the current pass, there
is no adjacent index at which both values move right during that pass.
The same assertion holds when both move left.
\end{lemma}

\begin{proof}
Assume that both values move right in the current pass. For each
value, choose an earlier pass with exactly one exchange. If this
exchange was rightward, denote its time and index by \(\theta,q\).
By Lemma~\ref{lem:single-exchange}, the value moves only right from
\(\theta\) through the current pass: a leftward exchange in between
would prevent its current rightward motion. Every other occurrence
of index \(q\) was ineffective in the chosen pass.
If the chosen exchange was leftward, take instead its first subsequent
rightward exchange as \(\theta,q\). The preceding exchange was
leftward at index \(q\) at time \(\theta-P\), by
Lemma~\ref{lem:single-exchange}. Between them the value stays at
position \(q\), so again every other occurrence of this index is
permanently ineffective. In both cases the value moves only right
from \(\theta\) through the current pass; only the phase \(\theta\)
modulo \(P\) can act at index \(q\) after \(\theta\).

Suppose, toward a contradiction, that both values move right
across index \(j\) in the current pass. They do not cross one another in
this pass, since both move only right by
Lemma~\ref{lem:single-exchange}. Name them \(x,y\) so that \(x\)
is to the right of \(y\) throughout the pass. Let
\(\theta_x,q_x\) and \(\theta_y,q_y\) be chosen as above.
At every index crossed by both in the current pass, \(x\) passes
before \(y\). We shall prove that their current exchange times at
\(j\) differ by at least
\(P\), which is impossible within one pass.

First suppose that \(y\) is already to the left of \(x\)
immediately after \(\theta_x\). It is then at a position at most
\(q_x\). To reach \(j\ge q_x\), it must cross index \(q_x\)
rightward after \(\theta_x\). The first such time is at least
\(\theta_x+P\), because only that phase remains available there.
Take its first rightward crossings of \(q_x,\ldots,j\) after
\(\theta_x\). Their times are strictly increasing, even if \(y\)
also moves left in between. Apply
\eqref{eq:period-separation-propagation} to these times and the
actual rightward exchange times of \(x\). The first such crossing
of \(j\) by \(y\), and hence its current crossing, is at least
\(P\) later than the crossing of \(x\). This case includes an
exchange of \(x,y\) at \(\theta_x\) itself, after which \(y\)
is on the left.

It remains to consider the case where \(y\) is to the right of
\(x\) immediately after \(\theta_x\). They must then cross at a
time \(T>\theta_x\), with \(x\) moving right and \(y\) moving
left. Thereafter \(x\) stays to the right of \(y\).
Since \(y\) moves only right from \(\theta_y\) through the
current pass, \(\theta_y>T\).

Suppose first that \(\theta_y\) is the first rightward exchange
following a chosen sole leftward exchange. Write
\(v=\theta_y-P\) for the time of its last leftward exchange, also
at index \(q_y\). Then \(v\ge T\). Immediately after \(v\), the value
\(y\) is at position \(q_y\), and \(x\) is farther right. Put
\[
 k=\max\{q_x,q_y\}.
\]
The value \(x\) has crossed index \(k\) by time \(v\). For
\(k=q_y\ge q_x\), this follows from its position at that time;
for \(k=q_x>q_y\), it follows from
\(\theta_x<T\le v\). Its exchange time \(\tau_k\) therefore
satisfies \(\tau_k\le v\), whereas the rightward exchange time
\(\sigma_k\) of \(y\) satisfies
\(\sigma_k\ge\theta_y=v+P\). Since the common current index
\(j\) is at least \(k\),
\eqref{eq:period-separation-propagation} proves the required
separation at \(j\).

Finally suppose that \(\theta_y\) itself is the sole rightward
exchange in a chosen earlier pass \((B,B+P]\), at index \(q_y\).
The leftward exchange at \(T\) cannot occur during or after that
pass, so \(T\le B\). At time \(B\), the value \(x\) is already
to the right of \(y\) and moves only right. During this pass,
\(y\) stays at position \(q_y\) before \(\theta_y\) and at
position \(q_y+1\) afterward. Its current exchange at \(j\)
occurs in a strictly later pass, so \(j\ge q_y+1\). Put
\[
 k=\max\{q_x,q_y+1\}\le j.
\]
The value \(y\) cannot cross index \(k\) in its sole-exchange
pass. Its subsequent rightward crossing time \(\sigma_k\)
therefore satisfies \(\sigma_k>B+P\).
If \(x\) has crossed \(k\) by time \(B\), its time
\(\tau_k\le B\) gives \(\sigma_k>\tau_k+P\).

Otherwise \(k=q_y+1\), and \(x\) is at position \(k\) at time
\(B\). Indeed, at time \(B\), the value \(y\) is at position
\(q_y\) and \(x\) is strictly to its right; if \(k=q_x>q_y+1\),
then \(x\) has already crossed \(k\) at \(\theta_x<T\le B\). The value \(x\) must cross
\(k\) before \(\theta_y\); otherwise the sole rightward
exchange of \(y\) would move \(x\) left. Thus
\[
 B<\tau_k<\theta_y.
\]
If \(\sigma_k<\tau_k+P\), then
\(B<\sigma_k-P<\tau_k\). Throughout this interval \(x\)
waits at position \(k\), so the occurrence effective at
\(\sigma_k\) was ineffective at \(\sigma_k-P\), a
contradiction. Hence \(\sigma_k\ge\tau_k+P\) in this case too.
Applying \eqref{eq:period-separation-propagation} from \(k\) to
\(j\) completes the rightward case. Reflection of positions and
reversal of the values proves the leftward case.
\end{proof}

\begin{proof}[Proof of Theorem~\ref{thm:cutoff}]
Choose a shortest positive elementary factorization of \(U\), with
index list \(Q\). As in Theorem~\ref{thm:12}, every index occurs,
\(a_r=\ell(U^r)\), and \(t_k(Q)=e_k(A)\). Fix
\(r\ge\lfloor(n-1)/2\rfloor\), and put \(s=r+1\), so that
\(2s>n-1\).

Every value that exchanges at least twice in pass \(s\) must have
made exactly one exchange in a strictly earlier pass. Otherwise it
would have at least two exchanges in every earlier pass as well:
a pass without exchanges would freeze it. It would then have crossed
at least \(2s\) distinct values, although only \(n-1\) others exist.
Consequently, no value moves in both directions in pass \(s\): such
motion would involve at least two exchanges and contradict
Lemma~\ref{lem:single-exchange}.

Suppose that some adjacent index \(i\) is used at least twice in
pass \(s\), and take two consecutive exchanges at that index.
Let \(x\) and \(y\) be their respective rightward-moving values.
The first value \(x\) must leave position \(i+1\) before the
second exchange. Since no value moves both ways in this pass, it
must leave by a further rightward exchange, at index \(i+1\).
Thus \(x\) has at least two exchanges in the pass. The second
value \(y\) is not the value moved left by the first exchange,
which cannot move right later in the same pass. It must therefore
enter position \(i\) from \(i-1\) between the two exchanges:
entry from the right would require another exchange at index
\(i\), contrary to their consecutive choice. Hence \(y\) also
has at least two exchanges. At an endpoint index one of these
required movements is impossible, already excluding repetition.

By the preceding observation, each of \(x,y\) has made exactly
one exchange in a strictly earlier pass.
Lemma~\ref{lem:period-separation} now excludes their rightward
exchanges at the same index \(i\) in pass \(s\). Every index
used in this pass is consequently used exactly once.

An index \(i\) is used in pass \(r+1\) precisely when
\(t_{n-i}(Q)>r\). Indeed, if it is unused, all its occurrences
are permanently ineffective. The set of values in the first
\(i\) positions is then fixed and, by eventual decreasing order,
already consists of the largest \(i\) values at the start of
the pass. Conversely, this separation prevents any exchange at
index \(i\). Thus the number of exchanges in this pass is
\(\#\{k:1\le k<n,\ t_k(Q)>r\}\). It is also
\(a_{r+1}-a_r\), which proves \eqref{eq:cutoff}.
\end{proof}

Write
\[
 \delta=\ell(U)-n+1,
 \qquad
 D_A=\binom n2-\sum_{k=1}^{n-1}e_k(A).
\]
The preceding equality gives the sharp upper bound for the total deficit.

\begin{samepage}
\begin{corollary}
\label{thm:sharp-deficit}
Let $A=U^T\Delta U$ be a symmetric oscillatory matrix of order $n\ge2$,
where $U$ is unit upper triangular and $\Delta$ is positive diagonal.
Then
\begin{equation}
\label{eq:sharp-deficit}
 \binom n2-\sum_{k=1}^{n-1}e_k(A)
 \le
 \left\lfloor\frac{n-1}{2}\right\rfloor
 \bigl(\ell(U)-n+1\bigr).
\end{equation}
For every $n\ge3$, equality has a symmetric integer realization of
determinant one with $\ell(U)=n$; thus the coefficient is best possible.
For $n=2$, $\delta=D_A=0$.
\end{corollary}
\end{samepage}

\begin{proof}
Choose a shortest positive factor list $Q$ for $U$, put
$h=\lfloor(n-1)/2\rfloor$, and let $d_r$ be the nonnegative
nonincreasing excesses in Lemma~\ref{lem:13}. The identifications
$a_r=\ell(U^r)$ and $t_k(Q)=e_k(A)$ from Theorem~\ref{thm:12}
give $d_0=\delta$ and $D_A=\sum_{r\ge0}d_r$.
By Theorem~\ref{thm:cutoff}, $d_r=0$ for $r\ge h$. Hence
\[
 D_A=\sum_{r=0}^{h-1}d_r\le h\delta.
\]
For $n=2$, $h=0$ and $\ell(U)=1$, so $\delta=D_A=0$.
Assume $n\ge3$ for the sharpness construction.

For the upper coefficient, let
\[
 q=\left\lfloor\frac n2\right\rfloor,
 \qquad Q=(1,2,\ldots,n-1,q),
 \qquad A=U_Q^TU_Q,
\]
with every factor parameter equal to one.  This is a symmetric
oscillatory integer matrix of determinant one.  In the first
permutation pass the ascending list gives $(2,3,\ldots,n,1)$,
and the final comparator exchanges the increasing entries $q+1,q+2$.
Thus every comparator acts, and Lemma~\ref{lem:12} gives $\ell(U_Q)=n$.
The profile is
\begin{equation}
\label{eq:sharp-deficit-profile}
 t_k(Q)=
 \begin{cases}
 k,&k\le q,\\
 k-1,&k>q.
 \end{cases}
\end{equation}

For $k\le q$, successive ascending passes give $1^{k-r}0^{n-k}1^r$.
Before completion, the only $10$ after a pass is at $k-r<q$.
The final $q$-comparator is therefore inactive, and $t_k(Q)=k$.
Delete the first occurrence of $q$ from $Q$, obtaining the basic order
\[
 c=(1,\ldots,q-1,q+1,\ldots,n-1,q).
\]
Its direction string contains a single one, at position $q$.
For $k>q$, Theorem~\ref{thm:4} therefore gives $t_k(c)=k-1$, so
deletion monotonicity in Lemma~\ref{lem:6} gives $t_k(Q)\le k-1$.
These coordinatewise upper bounds give
$D_A\ge n-1-q=\lfloor(n-1)/2\rfloor$. Since $\delta=1$, the
bound already proved gives the reverse inequality. Thus $D_A=h$,
and equality of the sums forces equality at every coordinate, proving
\eqref{eq:sharp-deficit-profile}.

For this family the excesses $d_r$ from
Lemma~\ref{lem:13} are nonnegative nonincreasing integers with
$d_0=1$ and sum $h=\lfloor(n-1)/2\rfloor$. Consequently
$d_0=\cdots=d_{h-1}=1$ and $d_r=0$ for $r\ge h$.
The starting point in Theorem~\ref{thm:cutoff} therefore cannot be
lowered uniformly.
\end{proof}

Combining this with \eqref{eq:8} gives
$\delta\le D_A\le\lfloor(n-1)/2\rfloor\delta$.

\subsection{Complementary orders and ranks of powers}

We now bound complementary corner times individually and the ranks of
off-diagonal blocks of powers.

\begin{lemma}
\label{lem:noncentral-pair}
Let $Q$ be a list of length $L$ containing every index $1,\ldots,n-1$.
If $1\le k<n$ and $2k\ne n$, then
\begin{equation}
\label{eq:noncentral-pair}
 t_k(Q)+t_{n-k}(Q)\ge 2n-1-L.
\end{equation}
\end{lemma}

\begin{proof}
Put $h=n-k$ and assume $k<h$.  The exchange-grid predecessor relations
in the proof of Theorem~\ref{thm:4} hold also when $Q$ has repeated
indices.  The lower boundary in its $k$-particle grid and the upper
boundary in its $h$-particle grid are chronological event chains with
index sequences
\[
 P=(k,k-1,\ldots,1,2,\ldots,h),\qquad
 R=(h,h+1,\ldots,n-1,n-2,\ldots,k).
\]
For each chain, record its occurrence positions in $[L]$ and let
$\operatorname{des}$ count strict descents in this position sequence.  Each
descent requires a new pass, so
\[
 t_k(Q)+t_h(Q)\ge2+\operatorname{des}P+\operatorname{des}R.
\]

Cut the two chains at their turns $1,n-1$, join the increasing pieces
into $I$ and the decreasing pieces into $J$, and collapse a junction
if its two copies are the same occurrence position.  Let $\eta$ be
the number of uncollapsed junctions.  The labels of $I$ increase from
$1$ to $n-1$, with only $h$ possibly repeated; those of $J$ decrease,
with only $k$ possibly repeated.  These are counting sequences, not
new chronological event chains.  If $X,Y$ are their position sets and
$c=|X\cap Y|$, then
\[
 c\ge2n-2+\eta-L,
 \qquad
 \operatorname{des}I+\operatorname{des}J
 \le\operatorname{des}P+\operatorname{des}R+\eta.
\]
Because $k\ne h$, the common positions have distinct labels and occur
in opposite orders in $I,J$.  Between each consecutive pair of common
positions, the two opposite connecting segments contain a descent in
total.  These segments have disjoint interiors on each sequence; hence
$\operatorname{des}I+\operatorname{des}J\ge c-1$, also trivially when
$c\le1$.  Consequently
\[
 \operatorname{des}P+\operatorname{des}R
 \ge c-1-\eta\ge2n-3-L,
\]
which proves the assertion.
\end{proof}

\begin{lemma}
\label{lem:power-ranks}
Let $U$ be an $n$-by-$n$ unit upper triangular TN matrix, $n\ge2$,
with positive first superdiagonal.  Choose a shortest positive factor list $Q$, and let
$w$ be its first-pass permutation in Lemma~\ref{lem:12}.
Write $c(\pi)$ for the number of cycles of a permutation, including
fixed points, and put $\tau=(1\ n)$.  Then
\[
 g=\frac{\ell(U)-n+3-c(w^{-1}\tau)}2
\]
is a nonnegative integer, and
\begin{equation}
\label{eq:power-ranks}
 \operatorname{rank}U^r[1:i,i+1:n]\le r+g
 \qquad(1\le i<n,\ r\ge1).
\end{equation}
Here $r$ is an integer, and $a:b=\{a,\ldots,b\}$, with the empty
set allowed when $a>b$.
\end{lemma}

\begin{proof}
Put $L=\operatorname{inv}(w)=\ell(U)$ and
\[
 \rho_j=\#\{p\le j:w(p)>j\},\qquad
 P=\sum_{j=1}^{n-1}\rho_j.
\]
The greedy reachable set from $[a]$ is $w^{-1}([a])$.  Since an
independent set of suffix columns extends to a column basis of the
first $a$ rows, Lemma~\ref{lem:6} gives
\begin{equation}\label{eq:rank-profile}
 \operatorname{rank}U[1:a,b+1:n]
   =a-\#\{p\le b:w(p)\le a\}.
\end{equation}
In particular, $w$ has full support: $\rho_j\ge1$ for $1\le j<n$.
Call $j$ clean when $\rho_j=1$; the cuts $1,n-1$ are clean.

We use the same cycle count twice.  Let $K,F$ be the numbers of
nontrivial cycles and fixed points of $w$, and let $T$ count its
cycle peaks, namely vertices $v$ with
$w(v)<v$ and $w^{-1}(v)<v$.  Each nontrivial cycle has at least
one peak, so $K\le T$.  Also put
\[
 A_p=\#\{q<p:w(q)>w(p)\},\qquad
 B_p=\#\{q>p:w(q)<w(p)\}.
\]
The identities $B_p-A_p=w(p)-p$ and
$P=\sum_p(w(p)-p)_+$ give
\begin{equation}\label{eq:turn-count}
 L-P=\sum_p\min(A_p,B_p)\ge F.
\end{equation}
Indeed, a fixed point $p$ contributes $A_p=B_p=\rho_p\ge1$.
Full support excludes fixed points $1,n$.

Every cycle peak satisfies $\rho_{v-1}=\rho_v+1$.
Charge each peak other than $n$ to the distinct cut $v-1$,
whose capacity $\rho_{v-1}-1$ is at least one.  This gives
$T\le P-n+2$, and hence
\[
 c(w^{-1}\tau)\le K+F+1\le T+F+1\le L-n+3.
\]
Thus $g\ge0$; permutation parity shows that its defining numerator
is even.

Call the permutation edge $p\to w(p)$ upper when $p<w(p)$ and lower
when $p>w(p)$; it crosses cut $j$ when
$\min\{p,w(p)\}\le j<\max\{p,w(p)\}$.
For a clean cut $j$, direct counting in
\eqref{eq:rank-profile} gives two useful facts.  The sum
\begin{equation}\label{eq:rank-sum}
 \operatorname{rank}U[1:i,j+1:n]
 +\operatorname{rank}U[1:j,i+1:n]
\end{equation}
is at least $\rho_i+1$, and at most $|i-j|+2$.
For the lower bound, its difference from $\rho_i+\rho_j$ is the
number of permutation points in the square with coordinate interval
$(\min\{i,j\},\max\{i,j\}]$.
For the upper bound, the sum is $|i-j|$ plus the number of permutation
edges crossing both cuts $i,j$.  There are at most two such edges,
one upper and one lower.  In particular, if the unique upper edge
crossing $j$ does not cross $i$, the upper bound improves to
$|i-j|+1$.

We claim that some clean $j$ makes \eqref{eq:rank-sum}
at most $g+2$. If $i$ is clean, take $j=i$. Otherwise let
$a<i<b$ be the nearest clean cuts on either side, and put
$I=\{a+1,\ldots,b-1\}$, $d=b-a-1$.  Use the same charging argument
with capacities $\rho_j-1-\mathbf1_{j\in I}$, now leaving the two
possible peaks $n,b$ uncharged.  Each other peak $v$ has capacity
at least one at $v-1$: outside $I$ this follows from full support;
inside $I$, $v\ne b$ implies $v\in I$ and $\rho_v\ge2$.
Consequently $T\le P-n+3-d$, so
\[
 c(w^{-1}\tau)\le K+F+1\le T+F+1
                 \le L-n+4-d.
\]
Hence $d\le2g+1$.  One of $a,b$ is therefore within distance $g$
of $i$, except when
\begin{equation}\label{eq:run-equality}
 d=2g+1,\qquad i-a=b-i=g+1.
\end{equation}
Outside this case the upper bound for \eqref{eq:rank-sum}
proves the claim.

In \eqref{eq:run-equality}, equality holds in the preceding
counting chain.  Thus $1,n$ lie in the same cycle $\mathcal C$,
$T=K$, and $F=L-P$.  Every nontrivial cycle has exactly one peak,
and is therefore increasing from its minimum to its maximum and
then decreasing back.  In particular, $\mathcal C$ has exactly one
upper edge across every cut.  Equality in
\eqref{eq:turn-count} gives $\min(A_p,B_p)=0$ at every
nonfixed position.  For an upper edge $p\to w(p)$, we have
$B_p-A_p=w(p)-p>0$, so $A_p=0$.  Upper edges consequently cannot be nested:
$q<p<w(p)<w(q)$ would contradict $A_p=0$.

The upper edges across the clean cuts $a,b$ belong to $\mathcal C$.
If both also crossed $i$, they would be the same edge, with source
at most $a$ and endpoint greater than $b$.  Since $\rho_i\ge2$,
another upper edge crosses $i$.  It belongs to another cycle, which
cannot cross either clean boundary; its two endpoints lie in
$(a,b]$.  This edge is nested inside the first, a contradiction.
At least one boundary upper edge therefore does not cross $i$.
The improved upper bound for \eqref{eq:rank-sum} at that
boundary is $(g+1)+1=g+2$, proving the claim.

Finally, for $n$-by-$n$ matrices $X,Y$, splitting the summation
index at $j$ gives the block-rank estimate
\[
 \operatorname{rank}(XY)[1:a,b+1:n]
 \le \operatorname{rank}X[1:a,j+1:n]
     +\operatorname{rank}Y[1:j,b+1:n]
 \qquad(0\le a,b,j\le n).
\]
Applying this estimate repeatedly at a clean cut $j$ gives, for $r\ge2$,
\[
 \operatorname{rank}U^r[1:i,i+1:n]
 \le (r-2)+\operatorname{rank}U[1:i,j+1:n]
             +\operatorname{rank}U[1:j,i+1:n].
\]
The $r-2$ intermediate terms are all
$\operatorname{rank}U[1:j,j+1:n]=1$.
Choose the clean $j$ just found to obtain the bound $r+g$.
For $r=1$, the lower bound $\rho_i+1$ for
\eqref{eq:rank-sum} gives $\rho_i\le g+1$ directly.
\end{proof}

\begin{proposition}
\label{prop:central-compensation}
Let $n=2m\ge4$, and let $Q$ be a shortest positive factor list for
an $n$-by-$n$ unit upper triangular TN matrix with positive first
superdiagonal.  Put $L=|Q|$ and $\delta=L-n+1$.  Then
\begin{equation}
\label{eq:central-compensation}
 t_1(Q)+t_{n-1}(Q)+t_m(Q)\ge3m-\delta.
\end{equation}
\end{proposition}

\begin{proof}
In the periodic repetition of $Q$, greedily select the earliest
embedding of $(1,\ldots,n-1)$ and the earliest embedding of
$(n-1,\ldots,1)$.  These are the actual one-particle and one-hole
trajectories.  Let $a_i,b_i\in\{1,\ldots,L\}$ be their respective
occurrence positions modulo $L$, and retain the union of these
positions to form a sublist $Q_E$.  Put $V=U(Q_E)$, with the
retained parameters, and $H=L-|Q_E|$.
Deleting the other positions preserves both trajectories.  Hence
\[
 t_1(Q_E)=t_1(Q),\qquad t_{n-1}(Q_E)=t_{n-1}(Q).
\]
Applying Lemma~\ref{lem:noncentral-pair} at $k=1$ to $Q_E$ gives
\begin{equation}
\label{eq:endpoint-budget}
 H\le L+t_1(Q)+t_{n-1}(Q)-(2n-1).
\end{equation}

Every block $V[1:i,i+1:n]$ has rank one.  For an internal cut $2\le i\le n-2$,
only $a_i,b_i$ can be occurrences of $i$ in $Q_E$.  If they coincide,
there is only one such occurrence.  If $a_i<b_i$, then $a_{i-1}$
cannot lie in $(a_i,b_i)$: the ascending greedy trajectory would
encounter $b_i$ before its selected $a_i$.  Starting cyclically after
$b_i$, this outside occurrence of $i-1$ is encountered before any
position in $(a_i,b_i)$, so $b_{i-1}$ is outside that interval as well.
If $b_i<a_i$, the symmetric argument excludes both selected
occurrences of $i+1$ between the two $i$'s.  Two particle exchanges
across cut $i$ require both an $i-1$ and an $i+1$ between them, to
refill position $i$ and clear position $i+1$.  Thus neither case
allows rank two, by Lemma~\ref{lem:6}.  At boundary cuts the rank
bound follows from size.  Every index remains in $Q_E$, so all these
ranks are exactly one.

The sublist $Q_E$ is itself shortest.  Every retained occurrence acts
in at least one pass of one of its two endpoint trajectories.  An ineffective position occurrence remains ineffective in every later
pass, as shown before Lemma~\ref{lem:single-exchange}.  If a retained occurrence were ineffective in the first
permutation pass of $Q_E$, it could never act in a threshold binary
trajectory.  This contradicts its selection.  Lemma~\ref{lem:12}
therefore proves the assertion.

Let $v$ be the first-pass permutation of $Q_E$.  Its cut counts are
all one. Each nontrivial cycle contributes at least one upper edge
across every cut between its smallest and largest vertices.  These intervals
cannot overlap, and full support cannot leave a gap between them.
Consequently $v$ has one nontrivial cycle containing $1,n$; all its
other $f$ vertices are fixed.  The cycle has exactly one upper edge across each
cut, so it increases from $1$ to $n$ and then decreases back to $1$.
At its nonfixed increasing positions there is no larger entry to
the left; at its nonfixed decreasing positions there is no smaller
entry to the right.  The
fixed positions each contribute one to the counting identity
\eqref{eq:turn-count}.  Hence
\[
 \operatorname{inv}(v)=n-1+f,\qquad
 c(v^{-1}\tau)=f+2,\qquad g(v)=0,
 \quad \tau=(1\ n).
\]
Here $g(v)$ denotes the expression in Lemma~\ref{lem:power-ranks}.

Insert the $H$ omitted letters back into $Q_E$.  For an intermediate
list $W$, use its ordinary adjacent-transposition product $\pi_W$
and the quantity
\[
 G(W)=\frac{|W|-n+3-c(\pi_W^{-1}\tau)}2.
\]
The intermediate list need not be shortest.  Inserting one letter
changes the ordinary product by a conjugate transposition, so the
cycle count changes by one and $G$ increases by zero or one.
Both endpoint lists $Q_E,Q$ are shortest, giving
\[
 g(w)=G(Q)\le G(Q_E)+H=H,
\]
where $w$ is the first-pass permutation of $Q$.

At $r=t_m(Q)$, the central $m$-by-$m$ block of $U(Q)^r$ has positive
determinant by Lemma~\ref{lem:6}.  Lemma~\ref{lem:power-ranks}
therefore gives $m\le t_m(Q)+g(w)\le t_m(Q)+H$.
Combining this with \eqref{eq:endpoint-budget} proves
\eqref{eq:central-compensation}.
\end{proof}

The complementary-order estimates also recover
\eqref{eq:sharp-deficit}. For odd $n$, sum
\eqref{eq:noncentral-pair} over $k<n/2$. For even $n=2m\ge4$,
use \eqref{eq:central-compensation} for orders $1,m,n-1$ and then
\eqref{eq:noncentral-pair} for $2\le k<m$.

\subsection{Realizable basic profiles}

We compare profiles coordinatewise.  A profile is \emph{maximal} within
a class if no other profile in the class has every coordinate at least as
large and one strictly larger.

\begin{corollary}
	\label{cor:1}
	A positive integer vector \((f_1,\ldots,f_{n-1})\) is the nontrivial
	exponent profile of a symmetric basic \(n\)-by-\(n\) oscillatory matrix
	if and only if
	\[
	f_k+f_{n-k}=n
	\quad(1\le k<n),
	\qquad
	|f_{k+1}-f_k|\le1
	\quad(1\le k<n-1).
	\]
	These are also exactly the maximal profiles among all symmetric
	oscillatory matrices of order \(n\).
	The number of distinct such profiles is
	\[
	\begin{cases}
	3^{m-1},&n=2m,\\
	2\cdot3^{m-1},&n=2m+1.
	\end{cases}
	\]
\end{corollary}

\begin{proof}
	By Theorem~\ref{thm:12}, every symmetric oscillatory profile is
	dominated by a symmetric basic profile.  A symmetric basic profile
	cannot be strictly increased coordinatewise within the symmetric class,
	since this would violate a complementary-order bound.
	It remains to prove the stated classification.
	Necessity follows from Lemma~\ref{lem:8}.
	Conversely, for every
	\(1\le k<\lfloor n/2\rfloor\), put
	\[
	\sigma_k=1+f_k-f_{k+1}\in\{0,1,2\}
	\]
	and choose the two direction bits \(d_k,d_{n-k-1}\) to have sum
	\(\sigma_k\).  These pairs are disjoint.  If \(n=2m+1\), set the remaining
	middle bit to \(d_m=f_m-m\in\{0,1\}\); if \(n=2m\), the complementary
	sum condition forces \(f_m=m\). Choose a factor order \(c\) realizing
	this direction string. Reading the nested central intervals in
	\eqref{eq:2}, from the center outward, gives
	\(e_k(A(c))=f_k\) for \(1\le k<n\).

	For even order there are \(m-1\) independent neighboring differences,
	each with three choices.  For odd order there are the same choices and two
	choices for the middle bit, giving the stated counts.
\end{proof}

For two factor orders \(c_-\) and \(c_+\), write \(t_k(c_\pm)\) for the
right-hand side of \eqref{eq:2} computed from the corresponding
direction string.

\begin{proposition}
	\label{prop:1}
	Let
	\[
	B=U(c_-)^T\Delta U(c_+),
	\]
	where \(\Delta\) is positive diagonal.  Then \(B\) is oscillatory.  Its
	lower and upper half-bandwidths are \(1+r_0(c_-)\) and
	\(1+r_0(c_+)\), respectively; those of \(B^{-1}\) are
	\(1+r_1(c_-)\) and \(1+r_1(c_+)\).  Moreover,
	\[
	e_k(B)=\max\{t_k(c_-),t_k(c_+)\}\qquad(1\le k<n).
	\]
\end{proposition}

\begin{proof}
Lemma~\ref{lem:5} gives oscillation and expresses the four bandwidths
as \(\ell_\uparrow(c_\pm)\) and \(\ell_\downarrow(c_\pm)\).
Since each index occurs once, these are \(1+r_0(c_\pm)\) and
\(1+r_1(c_\pm)\).  Lemma~\ref{lem:6} identifies the upper-right and
lower-left corner times with \(t_k(c_+)\) and \(t_k(c_-)\), respectively.
Corollary~\ref{thm:2} then gives their maximum.
\end{proof}

By Lemma~\ref{lem:7}, this formula covers every basic
oscillatory matrix.  Taking the maximum over \(k\) recovers the
triangular-exponent identity of \cite[Theorem~12]{FallatLiu}.

\begin{theorem}
	\label{thm:5}
	A positive integer vector \((f_1,\ldots,f_{n-1})\) is the nontrivial
	exponent profile of a basic \(n\)-by-\(n\) oscillatory matrix if and only if
	\begin{align*}
	f_k&\le\max\{k,n-k\},&&1\le k<n,\\
	|f_{k+1}-f_k|&\le1,&&1\le k<n-1,\\
	f_k+f_{n-k}&\ge n,&&1\le k<n.
	\end{align*}
\end{theorem}

\begin{proof}
	For a basic factorization, put
	\(a_k=t_k(c_-)\) and \(b_k=t_k(c_+)\).  Then
	\(f_k=\max\{a_k,b_k\}\).  Lemma
	\ref{lem:8} gives the envelope and
	Lipschitz bounds for \(a\) and \(b\); taking their coordinatewise maximum
	preserves both.  The complementary identity gives
	\[
	f_k+f_{n-k}
	=\max\{a_k,b_k\}+\max\{n-a_k,n-b_k\}
	=n+|a_k-b_k|\ge n.
	\]

	Conversely, suppose the three conditions hold.  For every mirror pair
	\(k<n-k\), define
	\[
	a_k=f_k,\qquad a_{n-k}=n-f_k,
	\]
	and
	\[
	b_k=n-f_{n-k},\qquad b_{n-k}=f_{n-k}.
	\]
	If \(n=2m\), the first and third conditions force
	\(f_m=m\); set \(a_m=b_m=m\).  Both vectors have positive entries,
	since \(1\le f_k\le n-1\), and each has complementary sum \(n\).
	Away from the center, their neighboring differences come from those of
	\(f\), possibly with signs reversed, and hence have size at most one.
	For even \(n=2m\), the steps adjacent to the center satisfy the same
	bound because \(a_m=b_m=f_m=m\).  For odd \(n=2m+1\), the bounds
	\(f_m,f_{m+1}\le m+1\) and \(f_m+f_{m+1}\ge2m+1\) give
	\(f_m,f_{m+1}\in\{m,m+1\}\); the central steps in \(a,b\) therefore
	also have size at most one.  Corollary~\ref{cor:1} realizes
	\(a\) and \(b\) by factor orders.

	On the left half, the complementary lower inequality gives
	\(b_k=n-f_{n-k}\le f_k=a_k\); on the right half the roles are reversed.
	Consequently \(f_k=\max\{a_k,b_k\}\) for every \(k\), and
	Proposition~\ref{prop:1} realizes \(f\).
\end{proof}

The run lengths also determine the possible bandwidth pairs in the basic class.

\begin{samepage}
\begin{corollary}
\label{cor:2}
A pair of positive integers \((p,q)\) occurs as
\[
 \bigl(\operatorname{bw}A,\operatorname{bw}A^{-1}\bigr)
\]
for a symmetric basic oscillatory matrix of order \(n\) if and only if
\[
 (p,q)=(1,n-1),\qquad (p,q)=(n-1,1),
\]
or
\[
 p,q\ge2,
 \qquad p+q\le n.
\]
Consequently, for \(p,q\ge2\), the least order in the symmetric basic
class is \(p+q\).
\end{corollary}
\end{samepage}

\begin{proof}
The two exceptional pairs come from the constant direction strings.  If
both symbols occur, longest zero- and one-runs of lengths \(p-1\) and
\(q-1\) are disjoint subintervals of a string of length \(n-2\).  Hence
\(p+q\le n\).

Conversely, if \(p,q\ge2\) and \(p+q\le n\), repeat the block
\[
 0^{p-1}1^{q-1}
\]
and truncate it to length \(n-2\).  It contains a complete block, and its
longest zero- and one-runs have the prescribed lengths.  Apply
Theorem~\ref{thm:4}.
\end{proof}

For \(n=2\), the direction string is empty, both bandwidths equal one,
and the only nontrivial exponent is \(e_1=1\).  All preceding formulas
include this case.

\section{Simultaneous matrix and inverse bandwidths}
\label{sec:4}

To describe the possible bandwidth pairs at order \(n\), put
\[
\mathcal P_n=\{(1,n-1),(n-1,1)\}\cup\{2,\ldots,n-1\}^2.
\]
For \(n=2\), this set consists of \((1,1)\).

\begin{theorem}
\label{thm:6}
Let \(n\ge2\).  There exists a symmetric oscillatory matrix \(A\) with
\[
(\operatorname{bw}A,\operatorname{bw}A^{-1})=(p,q)
\]
if and only if \((p,q)\in\mathcal P_n\).  Every admissible pair is
realized by a symmetric integer oscillatory matrix of determinant one.
\end{theorem}

\begin{proof}
Both \(A\) and \(A^\vee\) are oscillatory, so
\(1\le p,q\le n-1\).  If \(p=1\), then \(A\) is tridiagonal and
\[
(A^{-1})_{1n}=(-1)^{n+1}
\frac{a_{12}a_{23}\cdots a_{n-1,n}}{\det A}\ne0.
\]
Thus \(q=n-1\).  Applying this argument to \(A^\vee\) treats \(q=1\).
These are all necessary restrictions.

For sufficiency, let \((p,q)\in\mathcal P_n\).  If \(n\le p+q\), take
\begin{equation}
\label{eq:4}
w=(n-1,n-2,\ldots,n-q+1,\ 1,2,\ldots,p),
\end{equation}
where the first segment is empty for \(q=1\).  Since \(p\ge n-q\),
every index occurs.  The second segment contains an increasing sequence of
length \(p\).  Any increasing subsequence uses at most one index from the
first segment; if it then enters the second, its terminal value is at most
\(p\), so its length is at most \(p\).  Thus
\(\ell_\uparrow(w)=p\).

The first segment followed by the occurrence of \(n-q\) in the second
gives a decreasing sequence of length \(q\).  A decreasing subsequence
uses at most one index from the second segment, so
\(\ell_\downarrow(w)=q\).  Lemma~\ref{lem:5}
gives the required matrix \(A=U(w)^TU(w)\), with all parameters one.
This also treats the two pairs having an entry one.

If \(n>p+q\), admissibility forces \(p,q\ge2\).
Corollary~\ref{cor:2} supplies a symmetric basic realization of the pair.
Taking all its parameters and diagonal entries equal to one gives
integrality and determinant one by Lemma~\ref{lem:5}.
\end{proof}

\begin{corollary}
\label{cor:3}
For \((p,q)\in\mathcal P_n\), the smallest number of positive elementary
upper factors in a representation of an oscillatory matrix
\[
A=U(w)^T\Delta U(w),\qquad
(\operatorname{bw}A,\operatorname{bw}A^{-1})=(p,q),
\]
is
\[
L_{\min}(n;p,q)=\max\{n-1,p+q-1\}.
\]
Here only the factors in \(U(w)\) are counted.  The minimum is over the
matrices, the sequences, the positive parameters and the positive diagonal
matrices \(\Delta\).
\end{corollary}

\begin{proof}
Write \(L=|w|\). Every adjacent index must occur: otherwise the corresponding first
off-diagonal entry of \(A\) is zero by \eqref{eq:1}.
Hence \(L\ge n-1\).  Lemma~\ref{lem:5} also
gives increasing and decreasing consecutive subsequences of lengths
\(p\) and \(q\).  Their sets of occurrence positions intersect in at
most one position, since two common positions would require their values
to be both increasing and decreasing.  Thus \(L\ge p+q-1\).

The sequence \eqref{eq:4} has length \(p+q-1\) when
\(n\le p+q\).  The permutation construction has length \(n-1\) when
\(n>p+q\).  Both lower bounds are therefore attained.
\end{proof}

For the pairs \((1,n-1)\) and \((n-1,1)\), the length \(n-1\)
is the bidiagonal count in \cite[Section~3]{JohnsonOleskyDriessche},
applied to \(U(w)^T\) or its inverse, respectively.
For \(n=3\), the lengths \(2,2,3\) for the pairs
\((p,q)=(1,2),(2,1),(2,2)\), respectively, also follow from the triangular
classification in \cite[Theorem~11]{JohnsonOleskyDriessche}, applied to
\(U(w)^T\).

Consequently, among admissible pairs, a basic realization exists exactly
when \(p+q\le n\).  The minimum number of occurrences beyond the first
occurrence of each index is \(\max\{0,p+q-n\}\).

The two triangular sides can also be chosen independently.

\begin{corollary}
\label{cor:4}
There exists an oscillatory matrix \(A\) of order \(n\) satisfying
\[
\operatorname{bw}_{\pm}A=p_\pm,\qquad
\operatorname{bw}_{\pm}A^{-1}=q_\pm
\]
if and only if
\[
(p_-,q_-)\in\mathcal P_n,
\qquad (p_+,q_+)\in\mathcal P_n.
\]
Every admissible quadruple has an integer realization of determinant one.
\end{corollary}

\begin{proof}
Suppose \(p_+=1\).  Deleting the last row and first column of \(A\)
leaves a lower triangular matrix with diagonal
\(a_{12},\ldots,a_{n-1,n}\).  The cofactor formula used in
Theorem~\ref{thm:6} therefore gives \(q_+=n-1\), without a
symmetry assumption.  Transposition and checkerboard inversion give the
other necessary restrictions.  For sufficiency, choose the two sequences
\(w_-,w_+\) supplied by the proof of Theorem~\ref{thm:6} for
the respective pairs.  Lemma~\ref{lem:5} applied
to \(A=U(w_-)^TU(w_+)\), with all parameters one, proves the claim.
\end{proof}

\subsection{Profiles at fixed bandwidths}

The construction \eqref{eq:4} also maximizes all the minor-order
exponents simultaneously when \(p+q\ge n\).

\begin{theorem}
\label{thm:13}
Let \(A\) be a symmetric oscillatory matrix of order \(n\), with
\(p=\operatorname{bw}A\) and \(q=\operatorname{bw}A^{-1}\).
For \(1\le k<n\),
\begin{equation}
\label{eq:11}
 e_k(A)\le
 \min\{\max\{k,n-p\},\max\{n-k,n-q\}\}.
\end{equation}
If \(p+q\ge n\), the matrices
\(U(w)^T\Delta U(w)\), with \(w\) as in \eqref{eq:4},
arbitrary positive parameters and positive diagonal \(\Delta\),
have bandwidth pair \((p,q)\) and attain equality in \eqref{eq:11}
for every \(k\). They use the minimum \(p+q-1\) upper factors.
Taking all parameters and diagonal entries equal to one gives an
integer realization of determinant one.
\end{theorem}

\begin{proof}
Write \(A=U^T\Delta U\) and choose a positive elementary
factorization of \(U\). By Lemma~\ref{lem:5}, its index list contains
an increasing consecutive subsequence of length \(p\).
Retain these occurrences and one occurrence of each remaining index.
As in Theorem~\ref{thm:12}, the resulting symmetric basic matrix
\(B\) satisfies \(e_k(A)\le e_k(B)\). Its direction string contains
a consecutive block of \(p-1\) zeros.

Suppose \(2k\le n\). A block of \(p-1\) positions in
\([1,n-2]\) meets the central interval \([k,n-k-1]\) in at least
\(\min\{\max\{p-k,0\},n-2k\}\) positions. Indeed, the intersection
is smallest when the block is placed at an end, with at most \(k-1\)
positions before or after the central interval. Formula \eqref{eq:2}
therefore gives
\[
 \begin{split}
 e_k(A)&\le n-k-\min\{\max\{p-k,0\},n-2k\}\\
       &=\min\{n-k,\max\{k,n-p\}\}
        \le\max\{k,n-p\}.
 \end{split}
\]
For \(2k\ge n\), basic domination gives \(e_k(A)\le k\).
Thus \(e_k(A)\le\max\{k,n-p\}\) in every case.
Apply this bound to \(A^\vee\) at order \(n-k\).
Proposition~\ref{thm:3} gives
\(e_k(A)=e_{n-k}(A^\vee)\le\max\{n-k,n-q\}\), proving
\eqref{eq:11}.

For simultaneous equality, put \(a=n-p\) and \(b=n-q\);
then \(a+b\le n\). The bandwidth and factor-count assertions for
\(w\) follow from Theorem~\ref{thm:6} and Corollary~\ref{cor:3}.
It remains to bound its pass times from below.
The list \(w\) is a subsequence of a complete decreasing pass
\((n-1,\ldots,1)\) followed by a complete increasing pass
\((1,\ldots,n-1)\). After \(r\) such full passes, the initial string
\(1^k0^{n-k}\) is
\[
 0^r1^{k-r}0^{n-k-r}1^r
\]
until one central segment disappears. Its completion time is
\(\min\{k,n-k\}\). Deleting operations cannot decrease this time,
by Lemma~\ref{lem:6}, so \(t_k(w)\ge\min\{k,n-k\}\).

The rightmost one starts at position \(k\) and finishes at position
\(n\). If \(k\le p\), it must make the \(a\) rightward exchanges
at indices \(p,p+1,\ldots,n-1\). Indices larger than \(p\) occur only in the
decreasing segment, before the occurrence of \(p\) in the increasing
segment. This one can therefore make at most one of these exchanges
per pass. If \(k>p\), the same reasoning applies to the \(n-k\)
remaining exchanges, whose indices all lie in the decreasing segment. Hence
\(t_k(w)\ge\min\{a,n-k\}\).

The leftmost zero starts at position \(k+1\) and finishes at
position one. It must move left across \(\min\{k,b\}\) of the
indices \(1,\ldots,b\). These indices occur only in the increasing
segment, so at most one such move is possible per pass. Thus
\(t_k(w)\ge\min\{k,b\}\).
Combining the three lower bounds gives
\[
 \begin{split}
 t_k(w)&\ge
 \max\{\min\{k,n-k\},\min\{a,n-k\},\min\{k,b\}\}\\
 &=\min\{\max\{k,a\},\max\{n-k,b\}\}.
 \end{split}
\]
For the last identity, distributing the minimum on the second line
adds only the term \(\min\{a,b\}\) to the maximum on the first.
That term is redundant: if \(k\ge\min\{a,b\}\), it is bounded by
\(\min\{k,b\}\); otherwise \(a+b\le n\) gives
\(n-k>\min\{a,b\}\), and it is bounded by \(\min\{a,n-k\}\).
Since \(e_k(U(w)^T\Delta U(w))=t_k(w)\), this lower bound matches
\eqref{eq:11} at every order.
\end{proof}

\begin{lemma}
\label{lem:14}
Let $2\le a\le b$ and $n\ge a+b$ be integers.
Suppose that a binary string $d_1,\ldots,d_{n-2}$ contains
$a-1$ consecutive zeros and contains no $b$ consecutive ones.
Then, for $1\le k\le\lfloor n/2\rfloor$,
\[
 k+\#\{i:k\le i\le n-k-1,\ d_i=1\}
 \le n-a+1-\left\lfloor\frac{n-a}{b}\right\rfloor.
\]
An empty interval contributes zero.
\end{lemma}

\begin{proof}
Put $N=n-a$, $s=\lfloor N/b\rfloor$, and let $z$ count the zeros
in $[k,n-k-1]$.  It suffices to prove $k+z\ge a+s-1$.
Every interval of length $b$ contains a zero.

If $k\ge a$, then $z\ge\lfloor(n-2k)/b\rfloor$.
Since $b\ge2$, the integer function
$k+\lfloor(n-2k)/b\rfloor$ is nondecreasing.  Hence
\[
 k+z\ge a+\left\lfloor\frac{N-a}{b}\right\rfloor
       \ge a+s-1,
\]
where the last inequality uses $a\le b$.

Now let $k<a$ and fix a block of $a-1$ zeros.
Neither outside interval has enough positions to contain this block.
If the block is not contained in $[k,n-k-1]$, at least $a-k$ zeros
occur consecutively at one end of that interval. Removing exactly
$a-k$ zeros from that end leaves an interval of length $N-k$, so
\[
 z\ge a-k+\left\lfloor\frac{N-k}{b}\right\rfloor
   \ge a-k+s-1.
\]
Here $k<b$.  This also covers a block extending across both ends.

Otherwise the fixed block lies inside $[k,n-k-1]$.
After removing it, the two remaining intervals have total length
$N-2k+1$ and contain $z'=z-a+1$ zeros.
A string with $u$ zeros and no $b$ consecutive ones has length
at most $bu+b-1$.  Applying this to the two intervals gives
\[
 \begin{split}
 N&\le bz'+2b+2k-3\\
  &\le b(z'+k+1)-1,
 \end{split}
\]
since the difference in the second inequality is $(k-1)(b-2)$.
Thus $s\le z'+k$, and again $k+z\ge a+s-1$.
\end{proof}

\begin{theorem}
\label{thm:14}
Fix $n\ge2$ and $(p,q)\in\mathcal P_n$, and put
$a=\min\{p,q\}$ and $b=\max\{p,q\}$.
Among symmetric oscillatory matrices of order $n$ with
$\operatorname{bw}A=p$ and $\operatorname{bw}A^{-1}=q$,
\begin{equation}
\label{eq:13}
 \max_A e(A)=
 \begin{cases}
 n-a+1-\left\lfloor\dfrac{n-a}{b}\right\rfloor,
       &p+q\le n,\\[5pt]
 \max\{\lfloor n/2\rfloor,n-a\},
       &p+q\ge n.
 \end{cases}
\end{equation}
Both formulas give $b$ when $p+q=n$.
The maximum has a symmetric integer realization of determinant one
using the minimum $\max\{n-1,p+q-1\}$ upper factors.
\end{theorem}

\begin{proof}
When $p+q\ge n$, Theorem~\ref{thm:13} supplies a single realization
whose profile is
\[
 F_k=\min\{\max\{k,n-p\},\max\{n-k,n-q\}\}.
\]
The maximum of these values is
$\max\{\lfloor n/2\rfloor,n-a\}$.
Indeed, on the left half $F_k\le\max\{\lfloor n/2\rfloor,n-p\}$,
and on the right half $F_k\le\max\{\lfloor n/2\rfloor,n-q\}$.
The endpoint values are $n-p,n-q$, and the value at
$k=\lfloor n/2\rfloor$ is at least $\lfloor n/2\rfloor$.
Its $p+q-1$ upper factors are minimal by Corollary~\ref{cor:3}.  Taking all parameters and diagonal entries
equal to one gives the required integer realization.

Now suppose $p+q<n$.  Admissibility gives $a\ge2$.  Put
\[
 H=n-a+1-\left\lfloor\frac{n-a}{b}\right\rfloor.
\]
Write $A=U^T\Delta U$ and choose a positive elementary upper
factorization of $U$, with index list $Q$.
Every adjacent index occurs in $Q$.
Deleting factors while retaining one occurrence of each index gives a
symmetric basic matrix whose profile dominates that of $A$, by the
Cauchy--Binet comparison in the proof of Theorem~\ref{thm:12}.
For $2k\le n$, retain an increasing consecutive index subsequence of
length $p$, together with one occurrence of each remaining index.
Such a subsequence exists by Lemma~\ref{lem:5}.
The resulting basic matrix $B$ has a direction string containing
$p-1\ge a-1$ consecutive zeros.
Its runs of ones have length at most $q-1\le b-1$, because deletion
cannot increase a decreasing consecutive index subsequence.
Theorem~\ref{thm:4} and Lemma~\ref{lem:14} therefore give
\[
 e_k(A)\le e_k(B)
 =k+\#\{i:k\le i\le n-k-1,\ d_i=1\}\le H.
\]

For $2k\ge n$, apply the preceding left-half bound to $A^\vee$
at order $n-k$.  Its bandwidths are $(q,p)$, so $a,b$ and $H$ are
unchanged.  Proposition~\ref{thm:3} gives
$e_k(A)=e_{n-k}(A^\vee)\le H$.  Thus $e(A)\le H$.

For attainment, first suppose $p=a\le q=b$ and write
$n-a=bs+r$, with $0\le r<b$.  Since $n\ge a+b$, we have $s\ge1$.
Choose the direction string
\[
 1^{b-1}0^{a-1}(1^{b-1}0)^{s-1}1^r.
\]
It has length $n-2$, longest zero-run $a-1$, longest one-run $b-1$,
and $n-a-s$ ones.
Realize it by a factor order $c$, set all its parameters to one,
and take $A=U(c)^TU(c)$.
Theorem~\ref{thm:4} gives the exact bandwidths $(a,b)$ and
\[
 e_1(A)=n-a+1-s=H.
\]
If $p=b>q=a$, interchange zeros and ones in the same construction;
then $e_{n-1}(A)=H$.
Both realizations are basic, and their $n-1$ upper factors are minimal
by Corollary~\ref{cor:3}.
Finally, at $p+q=n$ we have $n-a=b$, so both displayed formulas
equal $b$.
\end{proof}

For $p+q\le n$, the first value in \eqref{eq:13} is
$n-a-\lfloor(n-p-q)/b\rfloor$.
It improves the maximum $n-a$ allowed by \eqref{eq:11} by exactly
$\lfloor(n-p-q)/b\rfloor$.  The improvement is strict if and only if
$n\ge p+q+b$.

\begin{corollary}
\label{cor:8}
Fix \((p,q)\in\mathcal P_n\). Among symmetric oscillatory matrices
of order \(n\) with bandwidth pair \((p,q)\), the maximum of
\(\sum_{k=1}^{n-1}e_k(A)\) is \(\binom n2\) when \(p+q\le n\).
When \(p+q\ge n\), put
\[
 h=\max\left\{\left\lfloor\frac n2\right\rfloor,
                        n-\min\{p,q\}\right\}.
\]
Then
\begin{equation}
\label{eq:12}
 \max_A\sum_{k=1}^{n-1}e_k(A)
 =h(n-h)+\binom{n-p}{2}+\binom{n-q}{2}.
\end{equation}
At every feasible pair, one symmetric integer realization of determinant
one attains the maximum exponent sum, the maximum exponent, and the
minimum upper factor count simultaneously. At \(p+q=n\), both sum
formulas agree.
\end{corollary}

\begin{proof}
When \(p+q\le n\), the maximal-exponent realization in
Theorem~\ref{thm:14} is basic. Its exponent sum is \(\binom n2\),
the largest possible by Theorem~\ref{thm:12}, and its \(n-1\)
upper factors are the minimum by Corollary~\ref{cor:3}.
Thus it attains all three extrema simultaneously.

Now suppose \(p+q\ge n\), and take the simultaneous extremal
realization \(A\) from Theorem~\ref{thm:13}. It maximizes every
profile coordinate and uses the minimum upper factor count, so it
also attains all three extrema.
Put \(a=n-p\) and \(b=n-q\). Its profile is
\[
 e_k(A)=\min\{\max\{k,a\},\max\{n-k,b\}\}.
\]
Its maximum is $h$ by Theorem~\ref{thm:14}.

For an integer \(0\le r<h\), the condition \(e_k(A)\le r\) means
that \(k\le r\) and \(a\le r\), or that \(n-k\le r\) and \(b\le r\).
Each of these sets, when nonempty, has size \(r\).
The two sets are disjoint: if both occur, then \(r\ge a,b\),
and \(r<h\) forces \(r<\lfloor n/2\rfloor\).
Counting the exponents by their integer levels gives
\[
 \begin{split}
 \sum_{k=1}^{n-1}e_k(A)
 &=h(n-1)-\sum_{r=a}^{h-1}r-\sum_{r=b}^{h-1}r\\
 &=h(n-h)+\binom a2+\binom b2.
 \end{split}
\]
This proves \eqref{eq:12}. At \(p+q=n\), the construction is basic,
so the two sum formulas coincide. All constructions permit
parameters and diagonal entries one.
\end{proof}

\begin{example}
\label{ex:2}
For $n=12$ and $p=q=3$, the bounds in \eqref{eq:11} allow
$e(A)\le9$, whereas Theorem~\ref{thm:14} gives the exact maximum
$\max_A e(A)=7$. The direction string $1100110110$ has longest zero- and
one-runs both equal to two and contains six ones. Its symmetric basic
realization therefore has bandwidth pair $(3,3)$ and $e_1(A)=7$.
\end{example}

\begin{example}
\label{ex:1}
For \(n=5\) and \(p=q=3\), the sequence \((4,3,1,2,3)\), with all
factor parameters and diagonal entries equal to one, gives
\[
A=\begin{pmatrix}
1&1&1&1&0\\
1&2&2&2&0\\
1&2&3&4&0\\
1&2&4&7&1\\
0&0&0&1&2
\end{pmatrix},\qquad
A^{-1}=\begin{pmatrix}
2&-1&0&0&0\\
-1&4&-5&2&-1\\
0&-5&9&-4&2\\
0&2&-4&2&-1\\
0&-1&2&-1&1
\end{pmatrix}.
\]
Both half-bandwidths equal three.  No basic realization exists in this
order; the five upper factors in this construction are the minimum.
Its profile is \((2,2,2,2)\), which is the largest at this bandwidth
pair by Theorem~\ref{thm:13}. The maximum exponent sum is eight.
\end{example}

For \(p,q\ge1\), set \(n=\max\{p,q\}+1\) and take
\[
w_{p,q}=(q,q-1,\ldots,1,2,\ldots,p).
\]
These least-order realizations use a different factor order from
\eqref{eq:4}.

\begin{proposition}
\label{thm:7}
Let \(A=U(w_{p,q})^T\Delta U(w_{p,q})\), with positive parameters and
positive diagonal \(\Delta\).  Its bandwidth pair is \((p,q)\), its
order is the least possible for that pair, and its \(p+q-1\) upper
factors are the minimum in this order. For \(1\le k<n\),
\[
 e_k(A)=\min\{\max\{k,n-p\},\max\{n-k,n-q\}\}.
\]
Thus this profile is the largest at the given order and bandwidth pair.
\end{proposition}

\begin{proof}
The longest increasing and decreasing consecutive subsequences have
lengths \(p,q\), respectively, so
Lemma~\ref{lem:5} gives the bandwidths.  Any
matrix with these bandwidths has order at least \(\max\{p,q\}+1\).
The factor count follows from Corollary~\ref{cor:3}.

Assume \(p\ge q\), so \(n=p+1\).  In the notation of
Lemma~\ref{lem:6}, it suffices to count passes of
\(q,q-1,\ldots,1,2,\ldots,p\) on \(1^k0^{n-k}\).
Repeating the operation at index one immediately has no further effect:
after \(10\) has become \(01\), the same operation does nothing.
Thus one pass may be split into the decreasing list \(q,\ldots,1\)
followed by the increasing list \(1,\ldots,p\).

If \(k\le q\), each pass moves a zero to the left end of the remaining
interval and a one to its right end.  After \(t\) passes the string is
\[
0^t1^{k-t}0^{n-k-t}1^t
\]
until one central segment disappears.  The time is \(\min\{k,n-k\}\).
If \(k>q\), the first \(k-q\) decreasing segments do nothing, while the
increasing segments leave
\(1^q0^{n-k}1^{k-q}\).  Subsequent passes place a zero on the left and
a one on the right.  The total time is
\[
(k-q)+\min\{q,n-k\}=\min\{k,n-q\}.
\]
Together these give \(\min\{k,n-\min\{k,q\}\}\).
Since \(n=p+1\), this is the asserted symmetric expression.

For the other case, put \(V=DU(w_{p,q})^{-1}D\).  Its factor sequence is
the reversal \(w_{q,p}\), and \(A^\vee=V\Delta^{-1}V^T\).
Lemma~\ref{lem:6} applies to this order of the
factors as well: both corner times use the upper sequence of \(V\).
Apply the preceding pass count with \(p,q\) interchanged and then use
\(e_k(A)=e_{n-k}(A^\vee)\). This gives the same symmetric expression.
The maximality assertions follow from Theorem~\ref{thm:13}.
\end{proof}

For \(p,q\ge2\), the least orders in the symmetric basic and unrestricted
symmetric classes are therefore \(p+q\) and \(\max\{p,q\}+1\).

\section{Banded matrices and complementary minors}
\label{sec:5}

For an \((p,q)\)-banded matrix \(B\) and equally sized index sets
\(I=\{i_1<\cdots<i_r\}\), \(J=\{j_1<\cdots<j_r\}\), call the minor
\emph{structurally trivial} if some paired position lies outside the band:
\[
j_\alpha-i_\alpha>q\quad\text{or}\quad i_\alpha-j_\alpha>p.
\]
Such a minor vanishes.  In the first case the first \(\alpha\) selected
rows can meet only the first \(\alpha-1\) selected columns in a nonzero
permutation term.  In the second case the last \(r-\alpha+1\) selected
rows can meet only the last \(r-\alpha\) selected columns.  Both are
impossible.

\begin{theorem}
\label{thm:8}
Let \(A\) be oscillatory and let \(p,q\in\{1,\ldots,n-1\}\). If \(A\) is
\((p,q)\)-banded, then, for \(1\le k<n\),
\[
\tau_k^{\mathrm{UR}}(A)\ge\left\lceil\frac{n-k}{q}\right\rceil,
\qquad
\tau_k^{\mathrm{LL}}(A)\ge\left\lceil\frac{n-k}{p}\right\rceil,
\]
and
\[
e_k(A)\ge\left\lceil\frac{n-k}{\min\{p,q\}}\right\rceil.
\]
If \(A^{-1}\) is \((p,q)\)-banded, then
\[
e_k(A)\ge\left\lceil\frac{k}{\min\{p,q\}}\right\rceil.
\]
For every pair \(p,q\), one oscillatory matrix attains both corner bounds
and the first exponent bound simultaneously at all minor orders.  Its
checkerboard inverse attains the inverse exponent bound simultaneously
at all minor orders.
\end{theorem}

\begin{proof}
The matrix \(A^m\) is \((mp,mq)\)-banded.  Its upper-right corner
minor of order \(k\) is structurally trivial when \(mq<n-k\), and its
lower-left corner minor is trivial when \(mp<n-k\).  This proves the
corner bounds and hence the first exponent bound.  Apply that bound to
\(A^\vee\) at order \(n-k\) to obtain the inverse bound.  The equality
cases follow from the next proposition and its checkerboard inverse.
\end{proof}

An \((p,q)\)-banded matrix is \emph{banded totally positive} if every
structurally nontrivial minor is positive
\cite{BranquinhoFoulquieManas}. It is TN and nonsingular, although it
may have structural zero minors. The product property in
\cite[Theorem~2.18]{BranquinhoFoulquieManas} has a direct
Cauchy--Binet proof, also allowing zero bandwidths.
Let \(B,C\) be banded totally positive with bandwidths
\((p_1,q_1)\), \((p_2,q_2)\). For index sets \(I,J\) satisfying
\(-p_1-p_2\le j_\alpha-i_\alpha\le q_1+q_2\), put
\[
 k_\alpha=\min\{\max\{j_\alpha,i_\alpha-p_1\},i_\alpha+q_1\}.
\]
The three sequences entering this median increase by at least one
with \(\alpha\), and so does their median. It lies between
\(i_\alpha\) and \(j_\alpha\), so \(K=\{k_1<\cdots<k_r\}\)
is an index set in \([n]\). Moreover,
\[
 -p_1\le k_\alpha-i_\alpha\le q_1,
 \qquad
 -p_2\le j_\alpha-k_\alpha\le q_2.
\]
Thus \(\Delta_{I,K}(B)\Delta_{K,J}(C)>0\) is a positive term in
Cauchy--Binet. Outside the summed bandwidths the minor is
structurally trivial. Consequently \(BC\) is banded totally positive
with bandwidths
\((\min\{n-1,p_1+p_2\},\min\{n-1,q_1+q_2\})\).

\begin{proposition}
\label{cor:5}
Let \(B\) be banded totally positive with lower and upper bandwidths
\(p,q\ge1\). Then \(B\) is oscillatory. For every integer \(m\ge1\)
and index sets \(I=\{i_1<\cdots<i_r\}\), \(J=\{j_1<\cdots<j_r\}\),
\begin{equation}
\label{eq:5}
\Delta_{I,J}(B^m)>0
\quad\Longleftrightarrow\quad
-mp\le j_\alpha-i_\alpha\le mq\quad(1\le\alpha\le r).
\end{equation}
In particular, for \(1\le k<n\),
\[
\tau_k^{\mathrm{UR}}(B)=\left\lceil\frac{n-k}{q}\right\rceil,
\quad
\tau_k^{\mathrm{LL}}(B)=\left\lceil\frac{n-k}{p}\right\rceil,
\quad
e_k(B)=\left\lceil\frac{n-k}{\min\{p,q\}}\right\rceil.
\]
The exponent of \(B\) is \(e(B)=\lceil(n-1)/\min\{p,q\}\rceil\).
\end{proposition}

\begin{proof}
Nonsingularity and positivity of the first off-diagonals give oscillation.
Iterating the product property proves \eqref{eq:5}.
For index sets satisfying the bounds in \eqref{eq:5}, the intermediate
sets can be written explicitly, for integers \(0\le t\le m\), as
\[
 X(t)=\bigl\{
 \min\{\max\{j_\alpha,i_\alpha-tp\},i_\alpha+tq\}
 :1\le\alpha\le r\bigr\},
\]
with \(X(0)=I\) and \(X(m)=J\).

The largest coordinate displacement between two \(k\)-subsets is
\(n-k\), attained by \(L_k,R_k\).  This proves the corner and exponent
formulas directly from \eqref{eq:5}.  The maximum over \(k\)
occurs at \(k=1\).
\end{proof}

For every \(1\le p,q<n\), such a matrix exists.  If
\(N=\sum_{i=1}^{n-1}E_{i,i+1}\), the product
\[
B=(I+N^T)^p(I+N)^q
\]
is banded totally positive by the product property, since its
bidiagonal factors are banded totally positive.  Its extreme band entries
are positive, so the bandwidths are exactly \(p,q\).

\begin{corollary}
\label{cor:6}
If \(B=A^\vee\) is banded totally positive with bandwidths \(p,q\ge1\),
let \(I,J\) be \(k\)-subsets and
write \(I^c=\{i'_1<\cdots<i'_{n-k}\}\) and
\(J^c=\{j'_1<\cdots<j'_{n-k}\}\).  For every integer \(m\ge1\),
\[
\Delta_{I,J}(A^m)>0
\quad\Longleftrightarrow\quad
-mp\le i'_\alpha-j'_\alpha\le mq
\quad(1\le\alpha\le n-k).
\]
Consequently,
\[
e_k(A)=\left\lceil\frac{k}{\min\{p,q\}}\right\rceil\quad(1\le k<n),
\qquad
e(A)=\left\lceil\frac{n-1}{\min\{p,q\}}\right\rceil.
\]
\end{corollary}

\begin{proof}
Jacobi's identity gives
\[
\Delta_{I,J}(A^m)=\det(A^m)\Delta_{J^c,I^c}(B^m).
\]
Apply Proposition~\ref{cor:5}; its row set is \(J^c\)
and its column set is \(I^c\).  The profile formula follows by
complementary-order duality.
\end{proof}

\subsection{Tridiagonal and Green matrices}

A tridiagonal matrix has zero entries when \(|i-j|>1\).  It is
irreducible tridiagonal when all entries on both first off-diagonals are
nonzero.

\begin{lemma}
\label{lem:9}
Every tridiagonal oscillatory matrix is banded totally positive with
bandwidths \((1,1)\).
\end{lemma}

\begin{proof}
Let \(S\) be tridiagonal and oscillatory. Its positive leading principal
minors give Gaussian elimination without
pivoting, \(S=L\Delta U\), where \(L,U\) are unit lower and upper
bidiagonal, \(\Delta\) is positive diagonal, and the first
off-diagonal entries of \(L,U\) are positive.
Each bidiagonal factor is banded totally positive: a nonzero
permutation term cannot contain crossed matches, so a structurally
nontrivial minor is the product of its positive paired entries.
The product property therefore gives the assertion. For an
individual minor, the intermediate set is
\(H=\{\min(i_\alpha,j_\alpha):1\le\alpha\le r\}\).
\end{proof}

\begin{corollary}
\label{thm:9}
If \(S\) is tridiagonal and oscillatory, then for every integer
\(m\ge1\) and equally sized sets \(I,J\),
\[
\Delta_{I,J}(S^m)>0
\quad\Longleftrightarrow\quad
|i_\alpha-j_\alpha|\le m\quad(1\le\alpha\le r),
\qquad e_k(S)=n-k\quad(1\le k<n).
\]
If \(A\) is oscillatory and \(A^{-1}\) is tridiagonal, then for
every integer \(m\ge1\) and \(k\)-subsets \(I,J\),
\[
\Delta_{I,J}(A^m)>0
\quad\Longleftrightarrow\quad
|i'_\alpha-j'_\alpha|\le m\quad(1\le\alpha\le n-k),
\qquad e_k(A)=k\quad(1\le k<n),
\]
where the primed indices enumerate the complements as above.
Both \(e(S)\) and \(e(A)\) equal \(n-1\).
\end{corollary}

\begin{proof}
Apply Lemma~\ref{lem:9} and
Proposition~\ref{cor:5} to \(S\), and apply
Corollary~\ref{cor:6} to \(A\).
\end{proof}

For the classical Green form, let \(p_i,q_i>0\) and
\[
G_{ij}=p_{\min(i,j)}q_{\max(i,j)},\qquad
0<\frac{p_1}{q_1}<\cdots<\frac{p_n}{q_n}.
\]
These matrices and their minors are discussed in
\cite{Pinkus,DelgadoPenaPena,OlshevskyStrangZhlobich}.
Set \(x_i=p_i/q_i\), \(x_0=0\), \(d_i=x_i-x_{i-1}>0\) and
\(D_q=\operatorname{diag}(q_1,\ldots,q_n)\).  With
\(L_{ij}=1\) for \(j\le i\) and zero otherwise,
\[
G=D_qL\operatorname{diag}(d_1,\ldots,d_n)L^TD_q.
\]
Here \(L=(I+E_{n,n-1})\cdots(I+E_{2,1})\) is TN and
\(L^{-1}=I-\sum_{i=1}^{n-1}E_{i+1,i}\).  Thus \(G\) is nonsingular
TN with positive first off-diagonals, and
\[
G^{-1}=D_q^{-1}L^{-T}\operatorname{diag}(d_1^{-1},\ldots,d_n^{-1})
L^{-1}D_q^{-1}
\]
is irreducible tridiagonal.  Corollary~\ref{thm:9} gives
\[
e_k(G)=k\quad(1\le k<n),\qquad e_n(G)=1.
\]
For \(n\ge3\), the minor with rows \(1,2\) and columns \(2,3\) is zero,
so \(G\) itself is not TP.

The strict Min matrices \((x_{\min(i,j)})\), with
\(0<x_1<\cdots<x_n\), are obtained by \(p_i=x_i,q_i=1\).
The strict Max matrices \((x_{\max(i,j)})\), with
\(x_1>\cdots>x_n>0\), are simultaneous row and column reversals of
Min matrices.  Such reversal preserves every minor sign and commutes with
matrix powers, so both classes have the same profile \((1,2,\ldots,n-1,1)\).

\section{Support patterns and endpoint exponents}
\label{sec:6}

We finish with entry-support consequences of the classical TN pattern
realization theorem and the corner criterion.

For a square matrix \(M=(m_{ij})\), let \(\vec G(M)\) be the directed
support graph with an edge \(i\to j\), \(i\neq j\), whenever
\(m_{ij}\neq0\).  If \(M\) is nonsingular TN, then its diagonal is positive.
Let \(\operatorname{dist}(i,j)\) denote the length of a shortest directed
path, with value infinity if none exists and \(\operatorname{dist}(i,i)=0\).
Consequently, for an oscillatory \(A\),
\[
 \tau_1^{\mathrm{UR}}(A)
 =\operatorname{dist}_{\vec G(A)}(1,n),
 \qquad
 \tau_1^{\mathrm{LL}}(A)
 =\operatorname{dist}_{\vec G(A)}(n,1).
\]
Indeed, a positive entry of \(A^m\) is equivalent to a support walk of
length at most \(m\), and positive diagonal entries supply stationary
steps.

\begin{definition}[Double-echelon pattern]
A zero-nonzero pattern is double echelon, in the sense of
\cite[Section~1.6]{FallatJohnson}, if its nonzero positions occur
consecutively in every nonempty row and column, and the initial and final
support indices are nondecreasing as the row or column index increases.
\end{definition}

For example, the pattern
\[
\begin{pmatrix}
+&+&0&0\\
+&+&+&0\\
+&+&+&+\\
0&+&+&+
\end{pmatrix}
\]
has row intervals \([1,2],[1,3],[1,4],[2,4]\).
Its column intervals satisfy the same monotonicity requirement.

We use the diagonal-containing case of the classical TN pattern
realization theorem.

\begin{lemma}[{\cite[Theorem~1.6.3 and Corollary~1.6.5]{FallatJohnson}}]
\label{lem:11}
Let \(\mathcal S\) be a square \((0,+)\)-pattern containing all diagonal
positions.  There is a nonsingular TN matrix whose positive entries are
exactly the \(+\)-positions of \(\mathcal S\) if and only if
\(\mathcal S\) is double echelon.
\end{lemma}

\begin{proposition}
	\label{thm:10}
	Let \(\mathcal S\) be an \(n\)-by-\(n\) zero-nonzero pattern, with \(+\)
	denoting a prescribed nonzero position.  The following are equivalent:
	\begin{enumerate}
		\item $\mathcal S$ is the support pattern of an oscillatory matrix;
		\item $\mathcal S$ is double echelon and every position satisfying
		$|i-j|\le1$ is present;
		\item $\mathcal S$ is the support pattern of the inverse of an
		oscillatory matrix.
	\end{enumerate}
Consequently, if \(A\) is oscillatory, both \(A\) and \(A^\vee\) have
double-echelon support, and
\[
a_{ij}>0,\qquad
(A^\vee)_{ij}=(-1)^{i+j}(A^{-1})_{ij}>0
\qquad (|i-j|\le1).
\]
Both support digraphs \(\vec G(A)\) and \(\vec G(A^{-1})\) contain the
bidirected path \(1\leftrightarrow2\leftrightarrow\cdots\leftrightarrow n\).
\end{proposition}

\begin{proof}
An oscillatory matrix is nonsingular TN, so its diagonal is positive and
Lemma~\ref{lem:11} makes its support double echelon.
Theorem~\ref{thm:1} gives positive first off-diagonals, proving (i)
implies (ii).  Conversely, Lemma~\ref{lem:11} realizes any pattern in
(ii) by a nonsingular TN matrix; its first off-diagonals are positive,
so it is oscillatory.

If an oscillatory $B$ realizes $\mathcal S$, then $A=B^\vee$ is
oscillatory and $A^{-1}=DBD$ has the same support, proving (iii).
Conversely, $A^{-1}$ and the oscillatory matrix $A^\vee$ have the same
support, proving (iii) implies (i).
Applying (ii) to $A$ and $A^\vee$ gives all the stated consequences.
\end{proof}

\begin{lemma}\label{lem:10}
Let \(A\) be nonsingular, let \(m\ge1\) be an integer, and put \(T=A^{-1}\).
If \(\operatorname{dist}_{\vec G(T)}(i,j)>m\), then
\[
\det\bigl((A^m)[\{j\}^c,\{i\}^c]\bigr)=0.
\]
\end{lemma}

\begin{proof}
A nonzero term in \((T^m)_{ij}\) would give a directed walk from \(i\)
to \(j\) of length at most \(m\), after omitting diagonal steps.
The distance assumption therefore gives \((T^m)_{ij}=0\).
Since \(T^m=(A^m)^{-1}\), the cofactor formula yields
\[
0=(T^m)_{ij}=(-1)^{i+j}
\frac{\det\bigl((A^m)[\{j\}^c,\{i\}^c]\bigr)}{\det(A^m)}.
\]
The denominator is nonzero because \(A\) is nonsingular.
\end{proof}

Thus a large directed distance in the inverse support graph gives an
explicit complementary minor witnessing failure of total positivity.

\subsection{Endpoint rigidity}

Fallat and Liu's maximal-exponent criterion
\cite[Theorem~10]{FallatLiu} uses vanishing corner minors of orders
one and \(n-1\); see \cite[Lemma~8]{FallatLiu}.
These two orders admit the following support characterizations.

An upper Hessenberg matrix satisfies \(a_{ij}=0\) for \(i>j+1\); a lower
Hessenberg matrix satisfies \(a_{ij}=0\) for \(j>i+1\).

\begin{theorem}
	\label{thm:11}
	Let \(A\) be an \(n\times n\) oscillatory matrix.  Then
	\[
	e_1(A)=n-1
	\quad\Longleftrightarrow\quad
	A\text{ is upper Hessenberg or lower Hessenberg}.
	\]
	Moreover,
	\[
	e_{n-1}(A)=n-1
	\quad\Longleftrightarrow\quad
	A^{-1}\text{ is upper Hessenberg or lower Hessenberg}.
	\]
\end{theorem}

\begin{proof}
	The support digraph of \(A\) contains both orientations of the natural
	path \(1,2,\ldots,n\).  Hence the directed distances from \(1\) to \(n\)
	and from \(n\) to \(1\) are at most \(n-1\).

	The first distance equals \(n-1\) if and only if there is no forward edge
	\(i\to j\) with \(j\ge i+2\).  Indeed, such an edge, together with the
	natural path before and after it, gives a route of length less than
	\(n-1\); conversely, in its absence every directed step can increase the
	index by at most one.  Thus
	\[
	\operatorname{dist}(1,n)=n-1
	\quad\Longleftrightarrow\quad
	A\text{ is lower Hessenberg}.
	\]
	Similarly,
	\[
	\operatorname{dist}(n,1)=n-1
	\quad\Longleftrightarrow\quad
	A\text{ is upper Hessenberg}.
	\]
	The corner formula at order one proves the first equivalence.

	By checkerboard duality,
	\[
	e_{n-1}(A)=e_1(A^\vee),
	\qquad A^\vee=DA^{-1}D.
	\]
	Conjugation by \(D\) preserves the zero pattern, so the first equivalence
	applied to \(A^\vee\) gives the second.
\end{proof}

\begin{corollary}
	\label{prop:2}
	Let \(A\) be symmetric and oscillatory.  If \(e_1(A)=n-1\), then \(A\)
	is irreducible tridiagonal and
	\[
	e_k(A)=n-k,
	\qquad k=1,\ldots,n-1.
	\]
	If \(e_{n-1}(A)=n-1\), then \(A^{-1}\) is irreducible tridiagonal and
	\[
	e_k(A)=k,
	\qquad k=1,\ldots,n-1.
	\]
\end{corollary}

\begin{proof}
	A symmetric upper or lower Hessenberg matrix is tridiagonal.  Theorem
	\ref{thm:11} therefore gives the asserted sparsity of
	\(A\), or of the symmetric matrix \(A^{-1}\), respectively.  Irreducibility
	follows from oscillation, and the profile formulas are
	Corollary~\ref{thm:9}.
\end{proof}

\section{Conclusion}

Among symmetric oscillatory matrices, the basic matrices are exactly
those with maximal profiles or, equivalently, maximum exponent sum.
For a symmetric oscillatory matrix \(A=U^T\Delta U\), the increments
of \(\ell(U^r)\) eventually equal the number of minor orders at which
\(A^r\) still has zero minors. The optimal starting point
is given by Theorem~\ref{thm:cutoff}; the earlier excesses sum to
the total deficit. Together with Theorem~\ref{thm:12},
Corollary~\ref{thm:sharp-deficit} gives the sharp two-sided estimate.
At each feasible matrix--inverse bandwidth pair \((p,q)\), one symmetric integer
realization of determinant one attains the maximum exponent, maximum
exponent sum and minimum factor count simultaneously; when $p+q\ge n$,
it also maximizes every profile coordinate.

The banded and inverse-banded support formulas describe complementary
families of minor-order profiles. For an oscillatory matrix \(A\),
the equalities \(e_1(A)=n-1\) and \(e_{n-1}(A)=n-1\) characterize
upper or lower Hessenberg form of \(A\) and \(A^{-1}\), respectively.
In the symmetric class, they determine the complete tridiagonal or
inverse-tridiagonal profile.

\end{document}